\documentclass[11pt,reqno]{amsart}
\usepackage {amsmath, amsfonts, amssymb, amscd, amsthm, eucal}

\usepackage{amsbsy}
\usepackage[all]{xy}
\SelectTips{cm}{}
\SilentMatrices
\usepackage[dvipsnames,usenames]{color}

\theoremstyle{plain}{
  \newtheorem{thm}{Theorem}[section]
  \newtheorem{cor}[thm]{Corollary}
  \newtheorem{lem}[thm]{Lemma}
  \newtheorem{prop}[thm]{Proposition}

  \newtheorem{sssection}[thm]{}}

\theoremstyle{definition}{
  \newtheorem{defn}[thm]{Definition}
 \newtheorem{notation}[thm]{Notation}
  }

\theoremstyle{remark}

\renewcommand{\subsubsection}{\sssection\rm}

\numberwithin{equation}{section}

\newcommand{\pt}{pt}

\newcommand{\Spec}{\operatorname{Spec}}
\newcommand{\Hom}{\operatorname{Hom}}

\newcommand{\rk}{\operatorname{rk}}

\newcommand{\op}{\mathrm{op}}

\newcommand{\SmOp}{\mathcal Sm\mathcal Op}
\newcommand{\Sm}{\mathcal Sm}

\newcommand{\Sp}{\mathnormal{Sp}}

\newcommand{\Aff}{\mathbf{A}}

\newcommand{\ZZ}{\mathbb{Z}}

\newcommand{\M}{\mathbf{M}}

\newcommand{\colim}{\operatornamewithlimits{colim}}

\newcommand \xra {\xrightarrow }

\newcommand \lra {\longrightarrow }
\newcommand \hra {\hookrightarrow }

\newcommand \xla {\xleftarrow }

\newcommand{\angles}[1]{\langle #1 \rangle}
\DeclareMathOperator{\HGr}{\mathnormal{HGr}}
\DeclareMathOperator{\USp}{\mathnormal{HU}}
\newcommand{\OO}{\mathcal{O}}
\newcommand{\parens}[1]{\textup{(}#1\textup{)}}

\newcommand{\veps}{\varepsilon}

\DeclareMathOperator{\thom}{\mathnormal{th}}
\newcommand{\onto}{\twoheadrightarrow}
\DeclareMathOperator{\coker}{coker}

\newcommand{\fhom}{\operatorname{\mathbf{hom}}_{\bullet}}

\DeclareMathOperator{\hocolim}{hocolim}
\newcommand{\finite}{\text{\textit{fin}}}
\newcommand{\geom}{\text{\textit{geom}}}
\newcommand{\mono}{\rightarrowtail}
\newcommand{\shf}{\mathcal}
\newcommand{\ft}{\text{\textit{small}}}
\newcommand{\homog}{\text{\textit{homog}}}
\newcommand{\KO}{\mathbf{KO}}
\newcommand{\BO}{\mathbf{BO}}
\newcommand{\KSp}{\mathbf{KSp}}
\newcommand{\BGL}{\mathbf{BGL}}
\newcommand{\NN}{\mathbb{N}}
\newcommand{\HH}{\mathsf{H}}
\newcommand{\GG}{\mathbb{G}}
\newcommand{\hh}{\mathsf{h}}

\begin{document}

\title{On the motivic commutative ring spectrum $\BO$}

\author{Ivan Panin}
\address{Steklov Institute of Mathematics at St.~Petersburg, Russia}
\author {Charles Walter}
\address{Laboratoire J.-A.~Dieudonn\'e \\ UMR 6621 du CNRS \\ Universit\'e de Nice -- Sophia Antipolis \\ 28 Avenue Valrose \\ 06108 Nice Cedex 02 \\ France}

\thanks{The first author gratefully acknowledge excellent working conditions and support provided by
Laboratoire J.-A. Dieudonn\'{e}, UMR 6621 du CNRS, Universit\'{e} de Nice  Sophia Antipolis, 
and by the RCN Frontier Research Group Project no. 250399 “Motivic Hopf equations" at University of Oslo.
}

\begin{abstract}
We construct an algebraic commutative ring $T$-spectrum $\BO$ which
is stably fibrant and $(8,4)$-periodic and such that on
$\SmOp/S$ the cohomology theory
$(X,U) \mapsto \BO^{p,q}(X_{+}/U_{+})$ and
Schlichting's hermitian $K$-theory functor
$(X,U) \mapsto KO^{[q]}_{2q-p}(X,U)$
are canonically isomorphic.
We use the motivic weak equivalence $\ZZ \times HGr
\xra{\sim} \KSp$ relating the infinite quaternionic Grassmannian to symplectic $K$-theory
to
equip $\BO$ with the structure of
a commutative monoid in the motivic stable homotopy category.
When the base scheme is $\Spec \ZZ[\frac 12]$,
this monoid structure and the induced ring structure on the cohomology theory
$\BO^{*,*}$ are the unique structures compatible with the
products
$$KO^{[2m]}_0(X) \times KO^{[2n]}_0(Y) \to KO^{[2m+2n]}_0(X \times Y).$$
on Grothendieck-Witt groups induced by the tensor product of symmetric chain complexes.
The cohomology theory is bigraded commutative
with the switch map acting on $\BO^{*,*}(T \wedge T)$ in the same way as
multiplication by the Grothendieck-Witt class of the symmetric bilinear space $\angles{-1}$.
\end{abstract}

\maketitle

\section{Introduction}

In a recent paper
\cite{Panin:2010fk}
we
defined motivic versions of symplectically oriented cohomology theories $A$ and of
quaternionic Grassmannians $HGr(r,n)$.  This $HGr(r,n)$ is the open subscheme of the ordinary
Grassmannian $Gr(2r,2n)$ parametrizing subspaces on which the standard symplectic form on
$\OO^{\oplus 2n}$ is nondegenerate.
We defined Borel classes of symplectic bundles in
such theories and calculated
\begin{equation}
\label{E:A(HGr)}
A(HGr(r,n)) = A(\pt)[b_{1},\dots,b_{r}]/(h_{n-r+1},\dots,h_{n})
\end{equation}
where the $b_{i}$ are the Borel classes of the tautological bundle on $HGr(r,n)$, and the
$h_{i}$ are the polynomials in the $b_{i}$ corresponding to the complete symmetric polynomials.
This is the same formula as the one which describes the cohomology of an ordinary Grassmannian
in terms of the Chern classes of an oriented cohomology theory.

In this paper we begin to apply those results to the hermitian $K$-theory of regular noetherian
separated schemes $X$ of finite Krull dimension with $\frac 12 \in \Gamma(X,\OO_{X})$.
We write $KO^{[n]}(X,U)$ for Schlichting's hermitian $K$-theory space for bounded complexes of
vector bundles on $X$ which are acyclic on the open subscheme $U \subset X$ and which are symmetric
with respect to the shift by $n$ of the usual duality.  We write $KO^{[n]}_{i}(X,U)$ for its homotopy
groups (for $i \geq 0$) or for Balmer's Witt groups $W^{n-i}(X,U)$ (for $i < 0$).

One of our main results is the following.

\begin{thm}
\label{T:SLc.ring}
For a regular separated noetherian scheme $S$ of finite Krull dimension with
$\frac 12 \in \Gamma(S,\OO_{S})$ Schlichting's hermitian $K$-theory is a
ring cohomology theory with an $SL^{c}$ Thom classes theory.
\end{thm}

Here \emph{ring cohomology theory} is used in the sense of \cite[Definitions 2.1 and 2.13]{Panin:2003rz}.
An \emph{$SL^{c}$ Thom classes theory} specifies a Thom class
$\thom(E,L,\lambda) \in KO_{0}^{[n]}(E,E-X)$ for every \emph{$SL_{n}^{c}$-bundle},
by which we mean a
triple $(E,L,\lambda)$ with
$E$ a vector bundle of rank $n$ over $X$, $L$ a line bundle and $\lambda \colon L \otimes L \to \det E$
an isomorphism.  These classes are functorial, multiplicative, and induces isomorphisms
$\cup \thom(E,L,\lambda) \colon KO_{i}^{[m]}(X) \to KO_{i}^{[m+n]}(E,E-X)$ for all $i$ and $m$.
The Thom classes restrict to Euler classes $e(E,L,\lambda) \in KO_{0}^{[n]}(X)$.
An $SL^{c}$ Thom classes theory gives Thom classes for all special linear, special orthogonal and
symplectic bundles.
So by the theory of \cite{Panin:2010fk} there are
Borel classes $b_{i}(E,\phi) \in KO_{0}^{[2i]}(X)$
for symplectic bundles.  The $b_{1}(E,\phi)$ of
a symplectic bundle
of rank $2r$
is the class corresponding to $[E,\phi] - r[\HH] \in KSp_{0}(X,U) = GW^{-}(X)$ under the
natural isomorphism $KSp \cong KO^{[2]}$.  Here $\HH$ is the trivial symplectic bundle of rank $2$.
The higher Borel classes will be calculated elsewhere.

We also construct several motivic spectra representing hermitian $K$-theory.
The first construction is a $T$-spectrum whose spaces
$(\KO^{[0]},\KO^{[1]},\KO^{[2]},\dots)$
are fibrant replacements
of presheaves composed of Schlichting's Waldhausen-like
hermitian $K$-theory spaces for bounded complexes of vector
bundles with shifted dualities \cite{Schlichting:2010uq}.
The structure maps $\KO^{[n]} \wedge T \to \KO^{[n+1]}$
are adjoint to the maps $\KO^{[n]}({-}) \to \KO^{[n+1]}({-} \wedge T)$
which are essentially multiplication by the Thom class
$\thom \in \KO_{0}^{[1]}(T)$ of the trivial line bundle.
Note the use of the appearance of the Thom classes in the very structure of the spectrum.
For $(X,U)$ in $\SmOp/S$ there are functorial isomorphisms
\begin{equation}
\label{E:isomorphisms}
KO_{i}^{[n]}(X,U) \cong \KO^{[n]}_{i}(X_{+}/U_{+}) \cong
\BO^{2n-i,n}(\Sigma_{T}^{\infty}(X_{+}/U_{+})),
\end{equation}
and the boundary maps $\partial \colon KO_{i}^{[n]}(U) \to KO_{i-1}^{[n]}(X,U)$
and
\[
\partial \colon \BO^{2n-i,n}(\Sigma_{T}^{\infty}(U_{+})) \to \BO^{2n-i+1,n}
(\Sigma_{T}^{\infty} (X_{+}/U_{+}))
\]
correspond.
%
Because it is based on the hermitian $K$-theory of chain complexes,
this spectrum has advantages in certain situations
over the one  constructed several years ago by Hornbostel
\cite{Hornbostel:2005ph}.
It treats all
shifts/weights uniformly instead of dealing in one way with the $K^{h}$-theory of the even
weights and in another way with the $U$-theory and $V$-theory of the odd weights.
It naturally handles non-affine schemes $X$ and even pairs
$(X,U)$ with $U \subset X$ open.  Finally we can easily identify the Thom, Euler
and Borel classes in the Grothendieck-Witt groups of chain complexes
$GW^{[n]}(X\ on\ X-U)$.

We show that the Morel and Voevodsky's theorem on Grassmannians and algebraic $K$-theory extends
to the symplectic context.  There are in truth only a few things to verify for symplectic groups
beyond what is in Morel and Voevodsky's paper.
(Orthogonal groups are much more problematic. {\bf However it has been done recently by M.Schlichting
and Shanker Tripathi.})

\begin{thm}
\label{T:MV.intro}
Let $HGr = \colim HGr(n,2n)$ be the infinite quaternionic Grassmannian.  Then
$\ZZ \times HGr$ and
$KSp$ are isomorphic in the motivic unstable homotopy category $H_{\bullet}(S)$.
\end{thm}

Next we define the $\times$ product.
The final group of theorems in the paper concerns the product structure.
The groups $KO_{0}^{[n]}(X,U)$, which are the
Grothendieck-Witt groups of bounded chain complexes of vector bundles which are symmetric
with respect to shifted
but untwisted dualities, have a naive product induced by the tensor product of chain complexes
\begin{equation}
\label{E:naive}
KO^{[m]}_0(X,U) \times KO^{[n]}_0(Y,V) \to
KO^{[m+n]}_0(X \times Y, X \times V \cup U \times Y).
\end{equation}
Let $\angles{1}$ and $\angles{-1}$ in $KO_{0}^{[0]}(\pt)$  denote the Grothendieck-Witt classes
of the rank one symmetric bilinear forms.
The product is \emph{$\angles{-1}$-commutative} meaning that
for $\alpha \in \BO^{p,q}(A)$ and $\beta \in \BO^{p',q'}(B)$
we have $\alpha \times \beta = (-1)^{pp'}\angles{-1}^{qq'} \sigma^{*}(\beta \times \alpha)$
where $\sigma \colon A\times B \to B \times A$ switches the factors.
Recall that a motivic space $A$ is called \emph{small}
if $Hom_{SH(S)}(\Sigma_{T}^{\infty}A, {-})$ commutes with arbitrary coproducts.

\begin{thm}
\label{uniq2}
The cohomology theory
$(\BO^{*,*}, \partial)$
on the category
$\M^{\ft}_{\bullet}(S)$
of small motivic spaces  over $S$
has a product $\times$
which is associative and $\angles{-1}$-commutative
with the unit $1 = \angles{1} \in \BO^{0,0}(\pt_{+})$,
which has $\alpha \times \Sigma_{\GG_{m}}1 = \Sigma_{\GG_{m}}\alpha$ and
$\alpha \times \Sigma_{S^{1}_{s}}1
= \Sigma_{S^{1}_{s}}\alpha$
for all $\alpha$, and
which restricts via the isomorphism
\eqref{E:isomorphisms}
to the naive ring structure
\eqref{E:naive}
on the groups
$KO^{[2n]}_{0}(X)$  for $X \in \Sm/S$.
It is the unique product with these properties.
\end{thm}

This is Theorems \ref{monoidBO} and \ref {T:unique.prod}.
Restricting to pairs $(X,U)$ with $U \subset X$ an open subscheme of a scheme smooth over $X$,
we get the following result.

\begin{thm}
\label{uniq1}
There is a canonical ring structure on the cohomology theory
$(KO^{[*]}_*, \partial)$
on $\SmOp/S$
which is associative and $\angles{-1}$-commutative
with unit
$\angles{1} \in KO^{[0]}_0(\pt)$
and which restricts
to the naive product on the Grothendieck-Witt groups
of chain complexes
$KO^{[2n]}_0(X)$.   This product and the
Thom classes of $SL^{c}$-bundles
make
$(KO_{*}^{[*]},\partial)$
ring cohomology theory with an $SL^{c}$ Thom classes theory.
\end{thm}

Our strongest result on the product is the following theorem.

\begin{thm}
\label{T:unique}
There exist morphisms $m \colon \BO \wedge \BO
\to \BO$
and $e \colon \Sigma_{T}^{\infty}\boldsymbol{1} \to \BO$ in $SH(S)$
which make $(\BO,m,e)$ a commutative monoid in $SH(S)$ and which are compatible
with the naive product in the following sense.
\begin{enumerate}
\item For all $X$ and $Y$ in $\Sm/S$ and all even integers $2p$ and $2q$ the naive products
\[
KO_{0}^{[2p]}(X) \times KO_{0}^{[2q]}(Y) \to KO_{0}^{[2p+2q]}(X \times Y)
\]
and the product
\[
\BO^{4p,2p}(X_{+}) \times \BO^{4q,2q}(Y_{+}) \to
\BO^{4p+4q,2p+2q}(X_{+} \wedge Y_{+})
\]
induced by $m$ correspond under the isomorphisms
\eqref{E:isomorphisms}.

\item The elements $\angles{1} \in KO_{0}^{[0]}(\pt)$ and $e \in \BO^{0,0}(S^{0,0})$ correspond under
the isomorphisms \eqref{E:isomorphisms}.
\end{enumerate}

Moreover, if for the base scheme $S$ the groups $KO_{1}(S)$ and $KSp_{1}(S)$ are finite
\parens{for example
$S = \Spec \ZZ[\frac 12]$},
then the monoid structure
$(m,e)$ with these properties is unique.
\end{thm}

For the proof of the theorem see Theorem \ref{T:unique.2}.

We explain the basic ideas in the proofs of these three theorems.
Gille and Nenashev's method \cite{Gille:2003ad} for constructing pairings in Witt groups of triangulated categories can be used
in hermitian $K$-theory to construct pairings between hermitian $K$-theory groups and
Gro\-then\-dieck-Witt groups
\begin{gather}
\label{E:partial.1}
KO^{[i]}_r(X,U) \times KO^{[j]}_0(Y,V) \to KO^{[i+j]}_r(X \times Y, X \times V \cup U \times Y)
\\
\label{E:partial.2}
KO^{[i]}_0(X,U) \times KO^{[j]}_r(Y,V) \to KO^{[i+j]}_r(X \times Y, X \times V \cup U \times Y)
\end{gather}
(see \eqref{E:KO.pairing.1}--\eqref{E:KO.pairing.2}).  They respect the boundary maps of the cohomology theorem $KO_{*}^{[*]}$.
This gives the hermitian $K$-theory groups the structure of a
cohomology theory with a partial multiplicative structure with Thom classes for all $SL^{c}$ bundles
including symplectic bundles.  Although this is less structure than we assumed
while writing \cite{Panin:2010fk}, it is enough to prove the quaternionic projective bundle
theorem and the symplectic splitting principle and to calculate the cohomology of quaternionic
Grassmannians.  Thus formula \eqref{E:A(HGr)} holds for $A = KO_{*}^{[*]}$ with the
$b_{i} \in KO_{0}^{[2i]}(HGr(r,n))$.
Because of the isomorphism \eqref{E:isomorphisms} it also holds for $A = \BO^{*,*}$
with the $b_{i} \in \BO^{4i,2i}(HGr(r,n))$.

The isomorphism \eqref{E:isomorphisms}
also transplants these pairings to $\BO$, giving in particular pairings
\begin{equation}
\label{BO2*BO**toBO**1}
\boxtimes \colon
\BO^{2i-r,i}(X_{+}) \times \BO^{2j,j}(Y_{+}) \to \BO^{2i+2j-r, i+j}(X_{+} \wedge Y_{+})
\end{equation}

Theorem \ref{T:MV.intro} and the isomorphisms $KSp \cong KO^{[4k+2]}$
gives us a canonical elements $\tau_{4k+2} \in \BO^{8k+4,4k+2}(\ZZ \times HGr)$.
Write $[-n,n] = \{ m \in \ZZ \mid -n \leq m \leq n \}$ and set
\[
HGr_{n} = [-n,n] \times HGr(n,2n).
\]
We have $\ZZ \times HGr = \colim HGr_{n}$.  A standard formula for homotopy colimits in triangulated categories
gives us an exact sequence
\[
0 \to \varprojlim\nolimits^{1} \BO^{8k+3,4k+2}(HGr_{n})
\to \BO^{8k+4,4k+2}(\ZZ \times HGr) \to
\varprojlim \BO^{8k+4,4k+2}(HGr_{n})
\to 0
\]
The inclusions $HGr_{n} \to HGr_{n+1}$ induce surjections on cohomology for any symplectically oriented
theory.  So the $\varprojlim\nolimits^{1}$ vanish, and we have
\begin{equation}
\label{E:ZHGr.lim}
\BO^{8k+4,4k+2}(\ZZ \times HGr) \cong
\varprojlim \BO^{8k+4,4k+2}(HGr_{n}).
\end{equation}
For essentially the same reasons we have isomorphisms
\begin{equation}
\label{E:ZHGr.ZHGr.lim}
\BO^{16k+8,8k+4}((\ZZ \times HGr) \wedge (\ZZ \times HGr))
\cong \varprojlim \BO^{16k+8,8k+4}(HGr_{n} \wedge HGr_{n}).
\end{equation}
Our class $\tau_{4k+2} \in \BO^{8k+4,4k+2}(\ZZ \times HGr)$ and the pairing
\eqref{BO2*BO**toBO**1} gives us a system of classes
\begin{equation}
\label{E:HGrn.HGrn}
\tau_{4k+2} |_{HGr_{n}} \boxtimes \tau_{4k+2} |_{HGr_{n}} \in \BO^{16k+8,8k+4}(HGr_{n} \wedge HGr_{n}).
\end{equation}
and therefore a class
\begin{equation*}
\tau_{4k+2} \boxtimes \tau_{4k+2} \in
\BO^{16k+8,8k+4}(\KO^{[4k+2]} \wedge \KO^{[4k+2]}).
\end{equation*}
There is also an exact sequence of the form
\[
0 \to \varprojlim\nolimits^{1} \BO^{4i-1,2i}(\KO^{[i]} \wedge \KO^{[i]}) \to
\BO^{0,0}(\BO \wedge \BO) \to
\varprojlim \BO^{4i,2i}(\KO^{[i]} \wedge \KO^{[i]}) \to 0.
\]
The elements $\tau_{4k+2} \boxtimes \tau_{4k+2}$ define an element
$\bar m \in \varprojlim \BO^{4i,2i}(\KO^{[i]} \wedge \KO^{[i]})$.
This $\bar m$ is the unique element compatible with the naive product \eqref{E:HGrn.HGrn}
and the isomorphisms \eqref{E:ZHGr.lim} and \eqref{E:ZHGr.ZHGr.lim}.
Lifting this element to $m \in \Hom_{SH(S)}(\BO \wedge\BO,\BO)$ gives an element we
can use to define the product $\times$.  On small motivic spaces $\times$ depends only on
$\bar m$ and not on the choice of $m$.
We deduce the associativity and bigraded commutativity
of $\times$ on small motivic spaces from the associativity and commutativity of the naive product on the
$KO_{0}^{[2r]}(HGr_{n})$.  This gives us Theorem \ref{uniq2}.

Theorem \ref{T:unique} is more subtle.  The obstructions to the uniqueness of $m$
and to the associativity, commutativity and unit property of the monoid it defines all live in certain
$\varprojlim^{1}$ groups.
We show in \S\ref{S:vanishing} that when $S = \Spec R$
with $\frac 12 \in R$ and with $KO_{1}(R)$ and $KSp_{1}(R)$ finite groups,
the $\varprojlim\nolimits^{1}$ vanish.  This uses the construction in \S\ref{S:finite} of three new
spectra using a new motivic sphere.

The geometry used to prove the quaternionic projective bundle theorem in \cite{Panin:2010fk}
also shows that the pointed quaternionic projective line
$(HP^{1},x_{0})$ is isomorphic to $T^{\wedge 2}$ in the motivic homotopy category
$H_{\bullet}(S)$.  The pointed scheme $HP^{1+}$ which is the $\Aff^{1}$ mapping cone
of the pointing $x_{0} \colon \pt \to HP^{1}$ is therefore also homotopy equivalent to $T^{\wedge 2}$.
It is the union of $HP^{1}$ and $\Aff^{1}$
with $x_{0} \in HP^{1}$ identified with $0 \in \Aff^{1}$, pointed at $1 \in \Aff^{1}$.
Therefore the motivic stable homotopy categories of $T$-spectra and of
$HP^{1+}$-spectra are equivalent.

We construct three $HP^{1+}$-spectra $\BO_{HP^{1+}}$, $\BO^{\geom}$
and $\BO^{\finite}$
representing hermitian $K$-theory.  The spaces of $\BO_{HP^{1+}}$ are the even-indexed
spaces $(\KO^{[0]}, \KO^{[2]}, \KO^{[4]}, \dots)$ of the $T$-spectrum.
The spaces of $\BO^{\geom}$ are alternately
$\ZZ \times RGr$ and $\ZZ \times HGr$.
The spaces of $\BO^{\finite}$ are finite unions of finite-dimensional real and quaternionic Grassmannians.
Here a \emph{real Grassmannian} $RGr(r,2n)$ is the open subscheme of the
ordinary Grassmannian $Gr(r,2n)$
where the hyperbolic quadratic form on $\OO^{\oplus 2n}$ is nondegenerate, while
$RGr = \colim RGr(n,2n)$.  For the details of $\BO^{\finite}$ see Theorem \ref{T:finite}.
The structure maps of the spectra are all essentially multiplication with the Euler class
$-b_{1}(\shf U)$ of the tautological rank $2$ symplectic subbundle on $HP^{1}$.
The bonding maps $\BO^{*}_{2i} \wedge HP^{1+} \to \BO_{2i+2}^{*}$ of the two geometric spectra
are morphisms of schemes or ind-schemes which are constant on the wedge $\BO^{*}_{2i} \vee HP^{1+}$.
The inclusion $\BO^{\finite} \to \BO^{\geom}$ is a motivic stable weak equivalence,
while the isomorphism $\BO^{\geom} \cong \BO_{HP^{1}}$ in $SH(S)$
is constructed from
\emph{classifying maps}
\begin{align*}
\tau_{4k} \colon \ZZ \times RGr \to \KO^{[4k]},
&&
\tau_{4k+2} \colon \ZZ \times HGr \xra{\sim} \KO^{[4k+2]},
\end{align*}
in $H_{\bullet}(S)$.  The $\tau_{4k+2}$ are the isomorphisms of Theorem \ref{T:MV.intro},
while the $\tau_{4k}$ are constructed from the $\tau_{4k+2}$.  We do not know if the
$\tau_{4k}$ are isomorphisms, although stably they have a right inverse (Proposition \ref{P:right.inverse}).

The $\varprojlim^{1}$ calculated with $\BO$ and with $\BO^{\finite}$ are the same.
Calculations based on the quaternionic projective bundle theorem show that for $\BO^{\finite}$
the $\varprojlim^{1}$ are over inverse systems of groups which are finite direct sums of
copies of $KO_{1}(S)$ and $KSp_{1}(S)$.  When those groups are finite, the $\varprojlim^{1}$ vanishes.
This gives Theorem \ref{T:unique} for $S = \Spec \ZZ[\frac 12]$.  For other $S$ one pulls the structure
back from $\Spec \ZZ[\frac 12]$ using the closed motivic model structure of \cite{Panin:2009aa}.

\begin{thm}
\label{T:compatible}
The products of Theorems \ref{uniq2}, \ref{uniq1} and \ref{T:unique} are compatible with all the
naive products of \eqref{E:naive} and with the partial multiplicative structure
of \eqref{E:partial.1} and \eqref{E:partial.2}.
\end{thm}

We do not know how to prove this theorem using our construction of the
hermitian $K$-theory product.  But Marco Schlichting has described to us (oral communication) how
to put a pairing on his hermitian $K$-theory spaces when one has a pairing
of complicial exact categories with weak equivalences and duality.  When applied to our situation
his product is isomorphic to ours on small motivic spaces by Theorem \ref{uniq2}.
Since Schlichting's product
is compatible with all the naive products of \eqref{E:naive} and with the partial multiplicative structure
of \eqref{E:partial.1} and \eqref{E:partial.2}, therefore ours is as well.

We finish the paper in \S \ref{S:K.theory}
by giving the analogue for algebraic $K$-theory of the spectra
$\BO^{\finite}$ and $\BO^{\geom}$ of hermitian $K$-theory of \S \ref{S:finite}.
The $\BGL^{\finite}$ seems to be completely new.  We write
$CGr(r,n)$ for the affine Grassmannian, $CP^{1} = CGr(1,2)$ for the affine version of
$\mathbf{P^{1}}$, and $CP^{1+}$ for the $\Aff^{1}$ mapping cone of the pointing map of $CP^{1}$.

\begin{thm}
There are $CP^{1+}$-spectra
$\BGL^{\finite}$ and $\BGL^{\geom}$ isomorphic to
$\BGL_{CP^{1+}}$ in $SH_{CP^{1+}}(S)$
with spaces
\begin{align*}
\BGL^{\finite}_{n} & =
[-4^{n},4^{n}] \times CGr(4^{n}, 2 \cdot 4^{n}),
&
\BGL^{\geom}_{n} & = \mathbb Z \times CGr,
\end{align*}
which are unions of affine Grassmannians.
The bonding maps $\BGL^{*}_{n} \wedge CP^{1+} \to \BGL_{n+1}^{*}$
of the two spectra
are morphisms of schemes or ind-schemes which are constant on the wedge
$\BGL^{*}_{n} \vee CP^{1+}$.
\end{thm}


The spectum $\BGL^{\finite}$ can be used to give alternate proofs of the uniqueness results
of \cite{Panin:2009aa} concerning the $\mathbf{P}^{1}$-spectrum representing algebraic $K$-theory and the
commutative monoid structure on that spectrum.  These proofs avoid the use of topological realization
and apply to any noetherian base scheme $S$ of finite Krull dimension with finite $K_{1}(S)$.

\section{Cohomology theories}
\label{S:cohom}

We fix a base scheme $S$ which is regular noetherian separated of finite Krull dimension and
with $\frac 12 \in \Gamma(S,\OO_{S})$.
The hermitian $K$-theory of such schemes is simpler than for other schemes, and we wish to
avoid the complications of negative hermitian $K$-theory and of characteristic $2$.

Let $\Sm/S$ be the category of smooth $S$-schemes
of finite type.
Let $\SmOp/S$ be the category whose objects are pairs $(X,U)$ with
$X$ in $\Sm/S$ and $U \subset X$ an open subscheme
and whose morphisms $f \colon (X,U) \to (Y,V)$ are morphisms $f \colon X \to Y$ of $S$-schemes
having $f(U) \subset V$. We write $X$ for $(X,\varnothing)$.
The base scheme itself will often be written as $S = \pt$.

A \emph{cohomology theory} on $\SmOp/S$
\cite[Definition 2.1]{Panin:2003rz}
is a pair
$(A,\partial)$ with $A$ a contravariant functor from $\Sm/S$ to the category of abelian groups
having localization exact sequences and satisfying \'etale excision and homotopy invariance.  The
$\partial$ is a morphism of functors with components
$\partial_{X,U} \colon A(U) \to A(X,U)$ which are the boundary maps of the localization
exact sequences.
A \emph{ring cohomology theory} in the sense of \cite[Definition 2.13]
{Panin:2003rz}
has products
\[
\times \colon A(X,U) \times A(Y,V) \to A(X\times Y, (X \times V) \cup (U \times Y))
\]
which are functorial, bilinear and associative and and which have a two-sided unit $1_{A} \in A(\pt)$ and
satisfy $\partial(\alpha \times \beta) = \partial \alpha \times \beta$.

A cohomology theory also defines groups $A(X,x)$ for pointed smooth schemes and their smash products
such as
\begin{equation}
\label{E:A(smash)}
A\bigl( (X_{1},x_{1}) \wedge (X_{2},x_{2}) \bigr)
= \ker\Bigl(A(X_{1}\times X_{2}) \xra{(x_{1}^{*} \times 1, 1 \times x_{2}^{*})}
A(\pt \times X_{2}) \oplus A(X_{1} \times \pt)
\Bigr).
\end{equation}
%

A \emph{bigraded} cohomology theory $(A^{*,*},\partial)$ is one in which the groups are
bigraded, that is
$A^{*,*}(X,U)= \bigoplus_{p,q \in \ZZ}
A^{p,q}(X,U)$, the pullback maps are homogeneous of bidegree $(0,0)$ and the boundary maps
$\partial_{X,U}$ are homogeneous of bidegree $(1, 0)$.
%
In a \emph{bigraded ring cohomology theory} $(A,\partial, \times,{1})$ the
$\times$ products respects the
bigrading and we have ${1} \in A^{0,0}(\pt)$.

\begin{defn}
\label{D:commutative}
Let $(A, \partial, \times ,{1})$ bigraded ring cohomology theory, and suppose
$\veps \in A^{0,0}(\pt) = A^{0,0}(\pt)$ satisfies $\veps^{2} ={1}$.
Then $(A, \partial, \times, {1})$ is  \emph{$\veps$-commutative} if for
$\alpha \in A^{p,q}(X,U)$ and $\beta \in A^{r,s}(Y,V)$ one has
$\sigma^{*}(\alpha \times \beta) = \beta \times \alpha \times (-1)^{pr}\veps^{qs}$
where $\sigma \colon Y \times X \to X \times Y$ switches the factors.
\end{defn}

Equivalently a bigraded ring cohomology is $\veps$-commutative if the associated cup product
satisfies $\alpha \cup \beta = (-1)^{pr}\veps^{qs} \, \beta \cup \alpha$ for $\alpha \in A^{p,q}(X,U)$
and $\beta \in A^{r,s}(X,V)$.
For such a cohomology theory the $A^{*,*}(X)$ are bigraded-commutative rings, and for any $(X,U)$
the $A^{*,*}(X,U)$ and $A^{*,*}(U)$ are right and left bigraded $A^{*,*}(X)$-modules.
The $\partial_{X,U}$ are morphisms of right $A^{*,*}(X)$-modules.
%

Sometimes it is easier to define certain products than others.
For a bigraded cohomology theory $(A^{*,*},\partial)$ set
$A^{0}(X,U) = \bigoplus_{p\in \ZZ} A^{2p,p}(X,U)$.
We need the following notion.

\begin{defn}
\label{multiplicative}
Let $(A^{*,*},\partial)$ be a bigraded cohomology theory as above. An \emph{$\veps$-commutative
partial multiplication} on $(A^{*,*},\partial)$ is given by
\begin{enumerate}

\item
pairings $\times \colon A^{p,q}(X,U) \times A^{2r,r}(Y,V) \to A^{p+2r,q+r}((X,U) \wedge (Y,V))$
which are bilinear and functorial, and
\item
elements ${\boldsymbol 1}$ and $\veps$ in $A^{0,0}(\pt)$
\end{enumerate}
satisfying
\begin{enumerate}
{\renewcommand{\theenumi}{\alph{enumi}}

\item
$\alpha \times (b \times c)= (\alpha \times b) \times c$ for $\alpha \in A^{p,q}(X,U)$,
$b \in A^{2r,r}(Y,V)$, $c \in A^{2s,s}(Z,W)$;
\item
$\alpha \times {\boldsymbol 1} = \alpha$ for  $\alpha \in A^{p,q}(Y,V)$,
\item
$\veps \times \veps = \boldsymbol{1}$,
\item
$a \times b = \sigma^*(b \times a) \times \veps^{rs}$ for  $a \in A^{2r,r}(X,U)$, $b \in A^{2s,s}(Y,V)$
where $\sigma\colon X \times Y \to Y \times X$
    switches the factors;
\item
$\partial_{Y \times X, V \times X}(\alpha \times b) = \partial_{Y,V}(\alpha) \times b$ for
$\alpha \in A^{p,q}(V)$, $b \in A^{2r,r}(X)$.
}\end{enumerate}
\end{defn}

If $(A^{*,*},\partial)$ has such a partial multiplication, then for
$\alpha \in A^{p,q}(X,V)$ and $b \in A^{2r,r}(X,U)$
one has a \emph{cup product}
\[
\alpha \cup b =
\Delta^{*}(\alpha \times b) \in A^{p+2r,q+r}(X,U \cup V).
\]

%
%
%
%
%

\label{A*istwosidedA0module}
If $(A,\partial)$ is equipped with a partial multiplicative structure
$(\times, \boldsymbol{1}, \veps)$, then the functor
$(X,U) \mapsto A^{0}(X,U)$ is an $\veps$-commutative graded ring functor in the sense that
the properties
(a), (b), (c) and (d) hold for $\alpha \in A^{0}(Y,V)$.

Moreover, $A$ is a bigraded
right $A^{0}$-module in the same sense with $\partial$ a morphism of bigraded right $A^{0}$-modules
which is homogeneous of bidegree $(1,0)$.

The switch $\sigma \colon X \times Y \to Y \times X$ allows us to define pairings
$\times \colon A^{2r,r}(X,U) \times A^{p,q}(Y,V) \to A^{p+2r, q+r}((X,U) \times (Y,V))$
by $b \times \alpha = \sigma^{*}(\alpha \times b) \times \veps^{qr}$.
There are also cup products $b \cup \alpha = \Delta^{*}(b \times \alpha)$.
The two pairings are compatible by (d).
Thus $A$ is a bigraded left and right $A^{0}$-module, with $\partial$ a morphism of right
$A^{0}$-modules.

\section{$SL$ and $SL^{c}$ orientations}
\label{S:SL.orientation}

We discuss $SL$ oriented cohomology theories.  Hermitian $K$-theory will turn out to be one.
We also include a discussion of Thom classes for vector bundles whose structural group is the double
cover $SL_{n}^{c}$ of $GL_{n}$.  It contains $SL_{n}$.  We believe this is the true level at which
Witt groups and hermitian $K$-theory are oriented.

An \emph{$SL$ bundle} on $X$ is a pair $(E,\lambda)$ with $E$ a vector bundle over $X$ and
$\lambda \colon \OO_{X} \cong \det E$ an isomorphism.  An \emph{isomorphism of $SL$ bundles}
$f \colon (E,\lambda) \cong (E_{1},\lambda_{1})$ is an isomorphism $f \colon E \cong E_{1}$ such
that $\lambda_{1} = \det f \circ \lambda$.

\begin{defn}
\label{D:SL.orientation}
An \emph{$SL$ orientation} on a bigraded
cohomology theory $A^{*,*}$ with an $\veps$-commu\-tative partial multiplication or ring structure
is an assignment to every $SL$ bundle $(E,\lambda)$ over every $X$ in $\Sm/S$
of a class $\thom(E,\lambda) \in A^{2n,n}(E,E-X)$ for $n = \rk E$
satisfying the following conditions:

\begin{enumerate}
\item For an isomorphism $f \colon (E,\lambda) \cong (E_{1},\lambda_{1})$
we have $\thom(E,\lambda) = f^{*}\thom(E_{1},\lambda_{1})$.

\item For $u \colon Y \to X$
we have $u^{*}\thom(E,\lambda) = \thom(u^{*}E,u^{*}\lambda)$ in $A^{2n,n}(u^{*}E,u^{*}E - Y)$.

\item The maps ${-} \cup \thom(E,\lambda) \colon A^{*,*}(X) \to A^{*+2n,*+n}(E,E-X)$ are isomorphisms.

\item We have
\[
\thom (E_{1} \oplus E_{2}, \lambda_{1} \otimes \lambda_{2})
= q_{1}^{*}\thom(E_{1},\lambda_{1}) \cup q_{2}^{*}\thom(E_{2},\lambda_{2}),
\]
where $q_{1},q_{2}$ are the projections from $E_{1} \oplus E_{2}$
onto its factors.
\end{enumerate}
The class $\thom(E,\lambda)$ is the \emph{Thom class} of the $SL$ bundle, and
$e(E,\lambda) = z^{*} \thom(E,\lambda) \in A^{2n,n}(X)$ is its \emph{Euler class}.
\end{defn}

This definition is analogous to the Thom classes theory version of the definition of an orientation
\cite[Definition 3.32]{Panin:2003rz}
or of a symplectic orientation
\cite[Definition 14.2]{Panin:2010fk}.

The Thom and Euler classes of $SL$ bundles are not necessarily central in contrast with
the classes in the oriented and symplectically oriented theories of
\cite{Panin:2003rz} and \cite{Panin:2010fk}.
But for an $SL$ bundle
of rank $n$ the Thom and Euler classes are in bidegree $(2n,n)$, and such classes need not be central when
$n$ is odd and $\veps \neq 1$.
Centrality occcurs for oriented theories because they have $\veps = 1$ and for
symplectically oriented theories because the Thom and Borel
classes of symplectic bundles are in bieven bidegrees
$(4r,2r)$.


Twisted versions of cohomology groups with coefficients in a line bundle
can be defined for any $SL$ oriented theory by
\begin{equation*}
A^{p,q}(X;L) = A^{p+2,q+1}(L,L-X)
\end{equation*}
and more generally by
$A^{p,q}(X,X-Z;L) = A^{p+2,q+1}(L,L-Z)$ for closed subsets $Z \subset X$.

\begin{thm}
\label{T:line.bundle}
Let $E$ be a vector of rank $n$ over $X$.  Suppose that $A^{*,*}$ is an
$SL$ oriented bigraded cohomology theory.
Then there are canonical isomorphisms of bigraded right $A^{0}(X)$ or
$A^{*,*}(X)$-modules  $A^{*+2n,*+n}(E,E-X) \cong A^{*,*}(X;\det E)$.
\end{thm}

%


\begin{proof}
Write $L_{E} = \det E$.
There are canonical isomorphisms
\begin{align*}
\lambda_{1} \colon \OO_{X} & \cong \det (E \oplus L_{E}^{\vee}), &
\lambda_{2} \colon \OO_{X} & \cong \det(L_{E} \oplus L_{E}^{\vee}).
\end{align*}
This gives us $SL$ bundles $(E \oplus L_{E}^{\vee},\lambda_{1})$ and
$(L_{E}\oplus L_{E}^{\vee}, \lambda_{2})$ over $X$.  The pullback of the first bundle
along $q \colon L_{E} \to X$ gives an $SL$ bundle whose structural map is the first projection
$L_{E} \oplus E \oplus L_{E}^{\vee} \onto L_{E}$.  The pullback of the second bundle along
$p \colon E \to X$ and permutation of the summands
gives an $SL$ bundle whose structural map is the second projection
$L_{E} \oplus E \oplus L_{E}^{\vee}
\onto E$.
We now have canonical isomorphisms
\[
\xymatrix @M=5pt @C=75pt {
A^{*+2n,*+n}(E,E-X) \ar[r]_-{\cong}^-{p^{*}\thom(L_{E}\oplus L_{E}^{\vee},\lambda_{2}) \cup}
&
A^{*+2n+4,*+n+2}(E \oplus L_{E} \oplus L_{E}^{\vee} ,E \oplus L_{E} \oplus L_{E}^{\vee} -X)
\ar[d]_-{\cong}^-{\veps^{n}}
\\
A^{*+2,*+1}(L_{E},L_{E}-X),
\ar[r]^-{\cong}_-{q^{*}\thom(E \oplus L_{E}^{\vee}, \lambda_{1}) \cup}
&
A^{*+2n+4,*+n+2}(L_{E} \oplus E \oplus L_{E}^{\vee} ,L_{E} \oplus E \oplus L_{E}^{\vee} -X)
}
\]
and the bottom left module is $A^{*,*}(X;\det E)$ by definition.  The sign $\veps^{n}$ is appropriate
when one permutes the rank $1$ bundles and the rank $n$ bundle.
\end{proof}

Hermitian $K$-theory and Witt groups have more Thom classes than just those for $SL$ bundles
because of what are often called periodicity isomorphisms such as
$W^{*}(X;L) \cong W^{*}(X;L \otimes L_{1}^{\otimes 2})$.  However, these periodicity isomorphisms
depend on choices.  A good way to structure these choices is to talk about $SL^{c}$ bundles,
using extra structure analogous to the $Spin^{c}$ structures frequently used in
differential geometry.

An \emph{$SL^{c}$ vector bundle} on $X$ is a triple $(E,L,\lambda)$ with $E$ a vector bundle,
$L$ a line bundle, and $\lambda \colon L \otimes L \cong \det E$ an isomorphism.  The structural
group of an $SL^{c}$ bundle of rank $n$ is $SL_{n}^{c}$ which is the kernel of
\[
GL_{n} \times \GG_{m} \xrightarrow{(\det^{-1}, \, t \,\mapsto \, t^{2})} \GG_{m}.
\]
There is a natural exact sequence $1 \to \mu_{2} \to SL_{n}^{c} \to GL_{n} \to 1$.   The notation
$SL^{c}$ is in imitation of $Spin^{c}$.  The role of this double cover of $GL_{n}$ in direct images
in real topological $K$-theory merited a mention by Atiyah
\cite[p.~55]{Atiyah:1971zr}
nearly forty years ago.

\begin{defn}
\label{D:SLc.orientation}
An \emph{$SL^{c}$ orientation} on a
cohomology theory $A^{*,*}$ with a $\veps$-commutative ring structure or partial multiplication
is an assignment to every $SL^{c}$ bundle $(E,L,\lambda)$ over every scheme $X$ in $\Sm/S$
of a class $\thom(E,L,\lambda) \in A^{2n,n}(E,E-X)$ where $n = \rk E$
satisfying the conditions (1)--(4) of Definition \ref{D:SL.orientation}.
%
%
%
%
\end{defn}

\section{Schlichting's hermitian K-theory and the Gille-Nenashev pairing}

In
\cite[\S 2.7]{Schlichting:2010uq}
Schlichting defines the hermitian $K$-theory space of a complicial exact category with
weak equivalences and duality in the style of Waldhausen's $K$-theory.  We will denote
his space by $KO(C,w,\sharp,\eta)$.  More generally we write
\[
KO^{[n]}(C,w,\sharp,\eta) = KO\bigl( (C,w,\sharp,\eta)[n] \bigr)
\]
for the hermitian $K$-theory space for the $n^{\text{th}}$ shifted duality,
and $KO_{i}^{[n]}(C,w,\sharp,\eta)$ for its homotopy groups.
A \emph{symmetric object of degree $n$} in $(C,w,\sharp,\eta)$ is a pair $(X,\phi)$ with
$\phi \colon X \to X^{\sharp}[n]$ a weak equivalence which is symmetric $\phi = \phi^{t}$
for the shifted duality.  There is a natural definition of a Grothendieck-Witt group
of symmetric objects of degree $n$, and
The $\pi_{0}$ of the hermitian $K$-theory space
is the Grothendieck-Witt group of degree $n$ symmetric objects
\[
KO^{[n]}_{0}(C,w,\sharp,\eta) = GW^{n}(C,w,\sharp,\eta).
\]
When $C$ is $\ZZ[\frac 12]$-linear,
that is the same as the triangulated Grothendieck-Witt group of the homotopy category
$Ho(C,w) = C[w^{-1}]$ for the duality $(\sharp,\eta)$, defined \`a la Balmer.

For a duality-preserving exact functor $(F,f) \colon (C,w,\sharp,\eta) \to (D,v,\natural,\varpi)$
there are induced maps of spaces $KO^{[n]}(C,w,\sharp,\eta) \to KO^{[n]}(D,v,\natural,\varpi)$.
A weak equivalence between duality-preserving exact functors $(F,f) \simeq (G,g)$ produces
a homotopy between the maps.

There are natural periodicity isomorphisms
$KO^{[n]}(C,w,\sharp,\eta) \simeq KO^{[n+4k]}(C,w,\sharp,\eta)$.
Moreover, we may write
\[
KSp^{[n]}(C,w,\sharp, \eta) = KO^{[n]}(C,w,\sharp,-\eta)
\]
because the effect of changing the sign is to interchange symmetric and skew-symmetric forms.
Then
there are
isomorphisms $KSp^{[n]}(C,w,\sharp,\eta) \simeq KO^{[n+4k+2]}(C,w,\sharp,\eta)$
induced by the duality preserving functor $X \mapsto X[2k{+}1]$.  However, it is more useful
to use the identifications
\begin{equation}
\label{E:KSp}
\begin{array}{ccc}
KSb_{i}^{[n]}(C,w,\sharp, \eta) & \lra & KO_{i}^{[n+4k+2]}(C,w,\sharp,\eta)
\\
\xi & \longmapsto & - \,\xi[2k{+}1]
\end{array}
\end{equation}
because these commute with the forgetful maps to Waldhausen's $K$-theory $K_{i}(C,w)$.

Among the many important results Schlichting proves is localization.  Suppose that
$\overline C_{1} \subset Ho(C,w)$ is a thick triangulated subcategory
which is stable under the duality.  Let $C_{1} \subset C$ be the full exact subcategory with the
same objects as $\overline C_{1}$.  Let $w_{1}$ be the set of all morphisms in $C$ whose mapping cone is
in $C_{1}$.

\begin{thm}
[\protect{\cite[Theorem 6]{Schlichting:2010uq}}]
\label{T:localization}
If $(C,w,\sharp,\eta)$ is a complicial exact category with weak equivalences and duality, and
$C_{1}$ is as above, then
\[
KO(C_{1},w,\sharp,\eta) \to KO(C,w,\sharp,\eta) \to KO(C,w_{1},\sharp,\eta)
\]
is a fibration sequence up to homotopy.
\end{thm}

Gille and Nenashev
\cite{Gille:2003ad}
have defined pairings for Witt groups of triangulated categories.
We explain how their construction can be applied to hermitian $K$-theory to give
pairings in the spirit of the partial multiplicative
structure of Definition \ref{multiplicative}.
A \emph{pairing}
\[
(\boxtimes, t_{1},t_{2},\lambda) \colon (C,w,\sharp,\eta) \times (D,v,\flat,\theta) \to
(E,u,\natural,\varpi)
\]
of complicial exact categories with weak equivalences and duality is an additive bifunctor
$\boxtimes \colon C \times D \to E$ which commutes with the translations up to specified functorial isomorphisms
$t_{1,X,Y} \colon X[1] \boxtimes Y \cong (X \boxtimes Y)[1]$ and $t_{2,X,Y} \colon X \boxtimes Y[1] \cong
(X \boxtimes Y)[1]$
plus
functorial weak equivalences
$\lambda_{X,Y} \colon X^{\sharp} \boxtimes Y^{\flat} \to (X \boxtimes Y)^{\natural}$
such that for any $X$ in $C$ and any $Y$ in $D$ the functors $X \boxtimes {-}$ and ${-} \boxtimes Y$
are exact and preserve weak equivalences and such that all the conditions of
\cite[Definitions 1.2 and 1.11]{Gille:2003ad}
hold.

Suppose given a symmetric object $(M,\phi)$ of degree $r$ in $(C,w,\sharp,\eta)$ and a symmetric object
$(N,\psi)$ of degree $s$ in $(D,v,\flat,\theta)$.  Gille and Nenashev show how to
define duality-preserving exact functors
\cite[Lemma 1.14]{Gille:2003ad}
%
\begin{subequations}
\begin{align*}
\bigl( {-} \boxtimes (N,\psi), \mathfrak R(N,\psi) \bigr)
\colon & (C,w,\sharp,\eta) \to (E,u,\natural, \varpi)[s]
\\
\bigl( (M,\phi) \boxtimes {-}, \mathfrak L(M,\phi) \bigr)
\colon & (D,v,\flat,\theta) \to (E,u,\natural,\varpi)[r]
\end{align*}
\end{subequations}
It follows that these induce maps of $KO$ spaces, which we will write as
\begin{subequations}
\begin{align}
\label{E:boxtimes.1}
( {-} \boxtimes (N,\psi))_{*}
\colon & KO^{[n]}(C,w,\sharp,\eta) \to KO^{[n]}(E,u,\natural, \varpi)[s]
\\
\label{E:boxtimes.2}
((M,\phi) \boxtimes {-} )_{*}
\colon & KO^{[n]}(D,v,\flat,\theta) \to KO^{[n]}(E,u,\natural,\varpi)[r]
\end{align}
\end{subequations}

The duality-preserving functor $(1_{E},-1)$,
which acts on symmetric objects by $(Z,\xi) \mapsto (Z,-\xi)$, induces a map
\begin{equation}
\label{E:involution}
\veps \colon KO^{[n]}(E,u,\natural, \varpi) \to KO^{[n]}(E,u,\natural, \varpi).
\end{equation}
These \emph{sign involutions} exist for the hermitian $K$-theory of any complicial exact category
with weak equivalences and duality, and they satisfy $\veps^{2}=1$ exactly.
(In general $\veps$ is not the same as the $-1$ which is the inverse map for the $H$-space structure
induced by the orthogonal direct sum.)
The methods of Gille and Nenashev show
\cite[Lemma 1.15]{Gille:2003ad}
that the
effect of the two functors on the Grothendieck-Witt classes $[M,\phi] \in GW^{r}(C,w,\sharp,\eta)$
and $[N,\psi] \in GW^{s}(D,v,\flat,\theta)$ is
\begin{equation}
\label{E:commute.category}
( {-} \boxtimes (N,\psi))_{*} [M,\phi]
= \veps^{rs}
( (M,\phi) \boxtimes {-})_{*}[N,\psi].
\end{equation}

\begin{prop}
The homotopy classes of the maps
\eqref{E:boxtimes.1} and \eqref{E:boxtimes.2}
on hermitian $K$-theory spaces
depend only on the classes $[N,\psi] \in GW^{s}(D,v,\flat,\theta)$
and $[M,\phi] \in GW^{r}(C,w,\sharp,\eta)$
respectively.
\end{prop}

\begin{proof}
There are three relations in the definition of the Grothendieck-Witt groups
\cite[Definition 1]{Schlichting:2010uq}.
The maps on homotopy groups are compatible with the relations
$[N,\psi] + [N_{1},\psi_{1}] = [N\oplus N_{1},\psi\oplus \psi_{1}]$ because
the orthogonal direct sum of symmetric objects induces a monoidal structure on the
$KO^{[n]}(E,u,\natural,\varpi)$ giving a naive additivity for orthogonal direct sums of
duality-preserving functors.

The maps are compatible with the relations $[N,\psi] = [N_{2},\sigma^{\flat}\psi \sigma]$ for
a weak equivalence $\sigma \colon N_{2}\to N$ because $\sigma$ induces a natural weak equivalence of
duality-preserving functors.

The maps are compatible with the third relation related to lagrangians
because of Schlichting's Additivity Theorem
\cite[Theorem 5]{Schlichting:2010uq}.
\end{proof}

We thus get pairings
\begin{subequations}
\begin{gather}
KO_{j}^{[m]}(C,w,\sharp,\eta) \times KO_{0}^{[s]}(D,v,\flat,\theta) \to KO_{j}^{[m+s]}(E,u,\natural,\varpi)
\\
KO_{0}^{[r]}(C,w,\sharp,\eta) \times KO_{i}^{[n]}(D,v,\flat,\theta) \to KO_{i}^{[n+r]}(E,u,\natural,\varpi)
\end{gather}
\end{subequations}
which we call the \emph{right pairing} and the \emph{corrected left pairing}.  They coincide on
$KO_{0}^{[m]} \times KO_{0}^{[n]}$.


The Gille-Nenashev pairings
have a number of other properties such as functoriality, associativity, compatibility
with localization sequences.  Using the right pairing means that the boundary map on the homotopy
groups in the localization sequences satisfies $\partial (\alpha \cup \xi) = \partial \alpha \cup \xi$
for $\alpha \in KO^{[n]}_{i}(C,w_{1},\sharp,\eta)$ and $\xi \in KO^{[r]}_{0}(D,v,\flat,\theta)$.

In
\cite{Schlichting:2010uq}
Schlichting constructs a hermitian $K$-theory space for $(C,w,\sharp,\eta)$ but not a spectrum.
So in principle what we have discussed so far does not yield any negative homotopy groups.
In practice it is known that
if the categories are $\ZZ[\frac 12]$-linear, then
Balmer's triangulated Witt groups for the homotopy category $Ho(C,w) = C[w^{-1}]$ can function
as negative homotopy groups
\[
KO^{[n]}_{i}(C,w,\sharp,\eta) = W^{n-i}(C[w^{-1}],\sharp,\eta)
\qquad \qquad \text{for $i < 0$}.
\]
This is explained in
\cite{Schlichting:2006aa}.
The localization sequences for the homotopy groups of the hermitian
$K$-theory spaces and Balmer's localization sequence for triangulated Witt groups
\cite[Theorem 6.2]{Balmer:2000hb}
attach to each other
because the $\pi_{0}$ are triangulated Grothendieck-Witt groups.

Two other important theorems
are the following.

\begin{thm}
[Fundamental Theorem \protect{\cite{Schlichting:2006aa}}]
\label{T:fundamental}
Let $(C,w,\sharp,\eta)$ be a $\ZZ[\frac 12]$-linear complicial exact category with weak
equivalences and duality. Then for all $n$
\[
KO^{[n-1]}(C,w,\sharp,\eta) \xra{F} K(C,w) \xra{H} KO^{[n]}(C,w,\sharp,\eta)
\]
is a homotopy fiber sequence, where $F$ is the forgetful map and $H$ the hyperbolic map.
\end{thm}

\begin{thm}
[\protect{\cite{Schlichting:2006aa}}]
\label{T:equivalence}
Let $(F,f) \colon (C,w,\sharp,\eta) \to (E,u,\natural, \varpi)$ be a duality-preserving
exact functor between $\ZZ[\frac 12]$-linear complicial exact categories with weak
equivalences and duality.  If $(F,f)$ induces a homotopy equivalence $K(C,w) \simeq K(E,u)$ of
Waldhausen $K$-theory spaces and isomorphisms
$W^{i}(C[w^{-1}],\sharp,\eta) \cong W^{i}(E[u^{-1}],\natural,\varpi)$ of Balmer's triangulated Witt groups,
then $(F,f)$ induces homotopy equivalences
$KO^{[n]}(C,w,\sharp,\eta) \simeq KO^{[n]}(E,u,\natural,\varpi)$ for all $n$.
\end{thm}

In particular if $(F,f)$ induces an equivalence of  $\ZZ[\frac 12]$-linear triangulated categories with
duality $(C[w^{-1}], \sharp, \eta) \simeq (E[u^{-1}],\natural,\varpi)$, then it induces a homotopy
equivalence $KO(C,w,\sharp,\eta) \simeq KO(E,u,\natural,\varpi)$ by Thomason's theorem
and by the fact that Balmer's Witt groups are a functorial over the category with objects
$\ZZ[\frac 12]$-linear
triangulated categories
with duality and arrows isomorphism classes of duality-preserving triangulated functors
\cite[Lemma 4.1]{Balmer:2002rp}.

\section{The cohomology theory KO on the category $\SmOp/k$ }

Let $S$ be a regular noetherian separated scheme of finite Krull dimension
with $\frac 12 \in \Gamma(S,\OO_{S})$. For every $S$-scheme $X$ consider the
category $VBX$
of big vector bundles over $X$ in the sense of  \cite[Appendix C.4]{Friedlander:2002aa}.
The assignments $X \mapsto VBX$
and $(f\colon Y \to X) \mapsto f^\ast\colon VBX \to VBY$ then form a
strict functor $(\Sm/S)^{\op} \to \mathcal{C}at$ because one has
 equalities $(f\circ g)^{*} = g^{*}\circ f^{*}$
instead of simply isomorphisms.

For any $X \in \Sm/S$, let $Ch^b(VBX)$ denote the additive category
of bounded complexes of big vector bundles on X. We will
consider $Ch^b(VBX)$ as a complicial exact category with weak
equivalences, the conflations being the degreewise-split short exact
sequences, and the weak equivalences $w_{X}$ being the quasi-isomorphisms.
When we further endow $Ch^b(VBX)$ with the duality consisting of the
functor ${}^\vee = \mathcal Hom_{\mathcal O_X}({-},\mathcal O_X)$ and
the natural biduality maps $\eta_{X} \colon 1 \cong {}^{\vee\vee}$, we will write it
as $Ch^b(VBX)$.
\[
Ch^{b}(VBX) = (Ch^{b}(VBX), w_{X}, {}^{\vee}, \eta_{X})
\]

Now suppose $U \subset X$ is an open subscheme and $Z = X -U$.
Let $w_{U}$ be the set of chain maps whose restriction to $U$ is a quasi-isomorphism.
Let $Ch^{b}(VBX)^{w_{U}}$ be the full additive subcategory of complexes which are
acyclic on $U$.  We have two new families of complicial exact categories with
weak equivalences and duality
\begin{align*}
Ch^{b}(VBX\ on \  Z) & = (Ch^{b}(VBX)^{w_{U}}, w_{X}, {}^{\vee}, \eta_{X}),
\\
Ch^{b}(VBX\ on \ U) & = (Ch^{b}(VBX),w_{U}, {}^{\vee}, \eta_{X}).
\end{align*}
We then have hermitian $K$-theory spaces
\begin{align*}
KO^{[n]}(X) & = KO^{[n]}(Ch^{b}(VBX) )
\\
KO^{[n]}(X,U) & = KO^{[n]}(Ch^{b}(VBX \ on \ Z)).
\end{align*}
with $KO^{[n]}(X) = KO^{[n]}(X,\varnothing)$.
Let
$D^{b}(VBX \ on \ Z)$ be the homotopy category equipped with the
triangulated duality $({}^{\vee},\eta_{X})$.
We define the hermitian
$K$-theory groups as
%
\begin{equation}
\label{E:KO.groups}
KO_{i}^{[n]}(X,U) =
\begin{cases}
\pi_{i}KO^{[n]}(X,U) & \text{for $i \geq 0$,}
\\
W^{n-i}(D^{b}VBX \ on \ Z) & \text{for $i < 0$}.
\end{cases}
\end{equation}

For $f \colon (X,U) \to (Y,V)$ a morphism in $\SmOp/S$ write $W = Y -V$.  Then the functor
$f^{*} \colon Ch^{b}(VBX \ on \ Z) \to Ch^{b}(VBY\ on \ W)$ can by made
duality-preserving by equipping it with the natural isomorphism
$f^{*}\mathcal{H}om_{\OO_{X}}({-},\OO_{X}) \cong \mathcal{H}om_{\OO_{Y}}(f^{*}{-},\OO_{Y})$.
This gives us maps
\begin{subequations}
\begin{gather}
\label{E:KO.pullback.1}
f^{*}\colon KO^{[n]}(X,U) \to KO^{[n]}(Y,V),
\\
\label{E:KO.pullback.2}
f^{*} \colon KO_{i}^{[n]}(X,U) \to KO_{i}^{[n]}(Y,V).
\end{gather}
\end{subequations}

By Schlichting's localization theorem (Theorem \ref{T:localization}) the sequences
\[
KO^{[n]}(X,U) \to KO^{[n]}(X) \to KO^{[n]}(Ch^{b}(VBX \ on \ U))
\]
are fibration sequences up to homotopy.
The restriction map $Ch^{b}(VBX \ on \  U) \to Ch^{b}(VBU)$ is a duality-preserving functor
which induces an equivalence on the homotopy categories $D^{b}(VBX \ on \ U) \simeq D^{b}(VBU)$.
(Schlichting actually gives a different argument in
\cite[\S 9]{Schlichting:2010uq}
which is valid with fewer restrictions on the schemes.)
So we get fibration sequences up to homotopy
\begin{subequations}
\begin{equation}
\label{E:localization}
KO^{[n]}(X,U) \to KO^{[n]}(X) \to KO^{[n]}(U)
\end{equation}
and therefore long exact sequences
\begin{equation}
\label{E:KO.localization}
\cdots \to KO_{i}^{[n]}(X,U) \to KO^{[n]}_{i}(X) \to KO^{[n]}_{i}(U)
\xrightarrow{\partial} KO_{i-1}^{[n]}(X,U) \to \cdots
\end{equation}
\end{subequations}

Now suppose $(Y,V)$ is in $\SmOp/S$ and write $W = Y - V$.
There is then a pairing of complicial exact categories with
weak equivalences and duality
\[
\boxtimes \colon
Ch^{b}(VBX \ on \ Z) \times Ch^{b}(VBY \ on \ W) \to Ch^{b}(VB(X \times Y) \ on \ Z \times W).
\]
For a degree $r$ symmetric complex $(N,\psi)$ on $Y$ which is acyclic on $V$ and for a
degree $s$ symmetric complex $(M,\phi)$ on $X$ which is acyclic on $U$ we have maps of
spaces
\begin{subequations}
\begin{align}
({-}\boxtimes (N,\psi))_{*} \colon &
KO^{[n]}(X,U) \to KO^{[n+r]}(X \times Y, (X \times V) \cup (U \times Y))
\\
\veps^{ns}((M,\phi) \boxtimes {-})_{*} \colon &
KO^{[n]}(Y,V) \to KO^{[n+s]}(X \times Y, (X \times V) \cup (U \times Y))
\end{align}
This leads to a right pairing and a corrected left pairing
\begin{gather}
\label{E:KO.pairing.1}
KO_{i}^{[n]}(X,U) \times KO_{0}^{[r]}(Y,V) \to KO_{i}^{[n+r]}(X \times Y, (X \times V) \cup (U \times Y)),
\\
\label{E:KO.pairing.2}
KO_{0}^{[n]}(X,U) \times KO_{i}^{[r]}(Y,V) \to KO_{i}^{[n+r]}(X \times Y, (X \times V) \cup (U \times Y)).
\end{gather}
\end{subequations}

Now suppose $(E,L,\lambda)$ is an $SL^{c}$ bundle of rank $n$ over $X$.  Let $p \colon E \to X$ be the
structural map.  We may construct Thom isomorphisms for hermitian $K$-theory using the
same method that Nenashev used for Witt groups \cite[\S 2]{Nenashev:2007rm}.
Namely, the pullback $p^{*}E = E \oplus E \to E$ has a canonical section $s$, the diagonal.
There is a Koszul complex
\[
K(E) = \bigl( 0 \to \Lambda^{n} p^{*}E^{\vee} \to \Lambda^{n-1} p^{*}E^{\vee} \to \cdots \to
\Lambda^{2} p^{*}E^{\vee} \to E^{\vee} \to \OO_{E} \to 0 \bigr)
\]
(considered as a chain complex in homological degrees $n$ to $0$)
in which each boundary map the contraction with $s$.    It is a locally free
resolution of the coherent sheaf $z_{*}\OO_{X}$.  There is a canonical isomorphism
$\varTheta(E) \colon K(E) \to K(E)^{\vee} \otimes \det p^{*}E^{\vee}[n]$
which is symmetric for the $(\det p^{*}E^{\vee})$-twisted shifted duality.
The composition
\[
\varTheta(E,L,\lambda) \colon K(E) \otimes p^{*}L \xra{\varTheta(E)}
K(E)^{\vee} \otimes \det p^{*}E^{\vee} \otimes p^{*}L [n]
\xra{1 \otimes p^{*}(\lambda^{\vee} \otimes L)}
K(E)^{\vee} \otimes p^{*}L^{\vee}[n]
\]
is symmetric for the untwisted shifted duality.  We consider
the Grothendieck-Witt class
\begin{equation}
\label{E:KO.Thom}
\thom(E,L,\lambda) = [K(E) \otimes p^{*}L, \varTheta(E,L,\lambda)] \in KO_{0}^{[n]}(E,E-X).
\end{equation}
When $L$ is trivial this is an $SL$ Thom class $\thom(E,\lambda) = [K(E), \varTheta(E,\lambda)]$.
Finally for $g$ a nowhere vanishing function on $X$ we let
$\angles{g} =  [\OO_{X},g] \in KO_{0}^{[0]}(X)$ be the Grothendieck-Witt class
of the rank one symmetric bilinear bundle.

\begin{thm}
\label{T:SLc.oriented}
Let $S$ be a regular noetherian separated scheme of finite Krull dimension with
$\frac 12 \in \Gamma(S,\OO_{S})$.  Then the groups $KO_{i}^{[n]}(X,U)$ of \eqref{E:KO.groups},
the maps $f^{*}$ of \eqref{E:KO.pullback.2} and $\partial$ of \eqref{E:KO.localization}, the pairings of
\eqref{E:KO.pairing.1} and \eqref{E:KO.pairing.2}, the classes $1 = \angles{1}$ and $\veps = \angles{-1}$
in $KO_{0}^{[0]}(\pt)$ and the classes $\thom(E,L,\lambda)$ of \eqref{E:KO.Thom} form
an $SL^{c}$ oriented cohomology theory with an $\veps$-commutative partial multiplication.
\end{thm}

In particular the products with the Thom classes are isomorphisms
\begin{equation}
\label{E:thom.iso.1}
{-} \cup \thom(E,L,\lambda) \colon KO_{i}^{[m]}(X) \xra{\cong}
KO_{i}^{[m+n]}(E, E-X)
\end{equation}

\begin{proof}
[Sketch of the proof]
The verifications are all straightforward.  For instance \'etale excision and homotopy invariance amount
to having certain pullback maps be isomorphisms, and pullback maps are induced by duality-preserving
functors.  Since these duality-preserving functors give isomorphisms for Quillen's $K$-theory and
Balmer's Witt groups, they give isomorphisms for hermitian $K$-theory as well under our hypotheses.

The Thom maps come from duality-preserving functors.   The functor part
$Ch^{b}(VBX) \to Ch^{b}(VBE)$ is given by
$\shf F \mapsto \pi^{*} \shf F\otimes_{\OO_{E}} K(E)$ with the target quasi-isomorphic to
the coherent sheaf
$z_{*}\shf F$.  These produce d\'evissage isomorphisms in both Quillen-Waldhausen $K$-theory
and Balmer's Witt groups \cite{Gille:2007hb}.

For the $\veps$-commutativity
of the partial multiplicative structure, $\sigma^{*}(b \times a)$ is calculated by applying the right
pairing for $a$ with respect to $b$ and then switching.  That is equivalent to applying the uncorrected
left pairing for $a$ with respect to $b$. But that satisfies \eqref{E:commute.category}.
\end{proof}

We will discuss later the Borel classes associated to the Thom classes.

It is sometimes inconvenient that the $KO_{i}^{[n]}(X,U)$ for $i < 0$ are not defined as homotopy groups.
But actually they are naturally isomorphic to direct summands of homotopy groups.
For we have the $S$-scheme $\GG_{m} = \Aff^{1}-0$
(pointed by $1$) and groups
$KO^{[n]}_i(\GG^{\wedge r}_m \times X,\GG^{\wedge r}_m \times U)$
defined as in \eqref{E:A(smash)}.


\begin{lem}
For all $i$ and $n$ and all $r \geq 1$ and all $(X,U)$ in $\SmOp/S$ there are natural isomorphisms
$KO_{i+r}^{[n+r]}(\GG^{\wedge r}_m \times X,\GG^{\wedge r}_m \times U)
\cong KO^{[n]}_{i}(X,U)$.
\end{lem}

\begin{proof}
For $r = 1$
we have a localization sequence which splits and a Thom isomorphism
\[
\xymatrix @M=5pt @C=12pt {
KO_{i+1}^{[n+1]}(\Aff^{1} \times X,\Aff^{1} \times U) \ar@{>->}[r]
& KO_{i+1}^{[n+1]}(\GG_{m} \times X,\GG_{m} \times U) \ar@{->>}[d]^-{\partial} \ar[ld]^-{(in_{1} \times 1_{X})^{*}}
&
\\
KO_{i+1}^{[n+1]}(X,U) \ar[u]_-{\cong}
&
KO_{i}^{[n+1]}(\Aff^{1} \times X, \Aff^{1} \times U \cup \GG_{m} \times X)
& KO_{i}^{[n]}(X,U) \ar[l]^-{\cong}_-{\thom \times}
}
\]
with $\thom \in KO_{0}^{[1]}(\Aff^{1},\Aff^{1} - 0)$
the Thom class of the trivial rank one $SL$ bundle.  Hence we have a natural
isomorphism
\[
KO_{i+1}^{[n+1]}(\GG_{m} \times X, \GG_{m} \times U) \cong
KO_{i+1}^{[n+1]}(X,U) \oplus KO_{i}^{[n]}(X,U).
\]
By induction $KO_{i+r}^{[n+r]}(\GG_{m}^{\times r} \times X, \GG_{m}^{\times r} \times U)$
is a direct sum of $2^{r}$ terms of which exactly one is
$KO^{[n]}_{i}(X,U)$.
\end{proof}

\begin{defn}
\label{D:periodicity}
The \emph{periodicity element} $\beta_{8} \in KO_{0}^{[4]}(\pt)$ is the element corresponding to
$1 \in KO_{0}^{[0]}(\pt)$ under the periodicity isomorphisms $KO_{i}^{[n]} \cong KO_{i}^{[n+4]}$
of the hermitian $K$-theory of chain complexes.
\end{defn}

Then for all $X$ and $n$ the periodicity isomorphisms $KO_{i}^{[n]}(X) \cong KO_{i}^{[n+4]}(X)$
coincides with ${-} \times \beta_{8}$ up to homotopy.

\section{$\KO^{[*]}_*$ of motivic spaces and $\KO^{[*]}_*(\ZZ \times HGr)$ }
\label{S:KO.motivic.spaces}

In this section we recall what the category of pointed motivic spaces is and
extend the functor $KO^{[*]}_*$ to a functor $\KO^{[*]}_*$ on that category.


The basic definitions, constructions and model structures we use
are given in \cite{Voevodsky:2007aa}.
A \emph{motivic space over $S$} is a simplicial presheaf on
the site $\Sm/S$ of smooth $S$-schemes of finite type.
A \emph{pointed motivic space over $S$} is a pointed simplicial presheaf on
the site $\Sm/S$.
We write
$\M_\bullet(S)$
for the category of pointed motivic spaces over $S$.


We equip the category $\M_\bullet(S)$ with the \emph{local injective model structure}
\cite[p.~181]{Voevodsky:2007aa} and with
the \emph{motivic model structure}
\cite[p.~194]{Voevodsky:2007aa}.  In both model structures the cofibrations are the monomorphisms.
The weak equivalences and fibrations of the local injective model structure are called
local weak equivalences and global fibrations.  Those of the motivic model structure are
called motivic weak equivalences and motivic fibrations.

We write $H_{\bullet}(S)$ for the pointed motivic
unstable homotopy category obtained by inverting the motivic weak equivalences.
%
The homotopy category $H_\bullet(S)$ is equivalent to the motivic homotopy category
of  \cite{Morel:1999ab}.
For a morphism
$f\colon A \to B$
of pointed motivic spaces
we will write
$[f]$
for the class of $f$ in
$H_\bullet (S)$.

\begin{notation}
\label{fibrantreplacement}
There is a global fibrant model functor
$G \colon Id_{\M_\bullet(S)} \to ({-})^{f}$ functor in $\M_\bullet(S)$.
The natural transformation $G$ is a local weak equivalence,
but we do not require it to be injective.
\end{notation}


\begin{lem}
\label{BbbKOfmotivicfibrant}
Let $\KO^{[i]}=(KO^{[i]})^{f}$.
Then the map
$G\colon KO^{[i]} \to \KO^{[i]}$
is a schemewise weak equivalence, and
the space $\KO^{[i]}$ is motivically fibrant.
\end{lem}

\begin{proof}
The space $KO^{[i]}$ satisfies Nisnevich descent
because the homotopy groups $KO^{[i]}_{r}$ satisfy \'etale excision.
Therefore
\cite[Theorem 5.21]{Voevodsky:2007aa}
the morphism
$G \colon KO^{[i]} \to \KO^{[i]}$
is a schemewise weak equivalence.
For every $X \in \Sm/S$ the projection
$\Aff^1_X \to X$ induces a weak equivalence of simplicial sets
$\KO^{[i]}(X) \to \KO^{[i]}(\Aff^1_X)$,
since this the case for the space $KO^{[i]}$
and
$G: KO^{[i]} \to \KO^{[i]}$
is a schemewise weak equivalence.
This proves \cite[p.~195]{Voevodsky:2007aa} that the globally fibrant space
$\KO^{[i]}$ is motivically fibrant.
\end{proof}

We write $S^{r}_{s}$ for the $r$-sphere $\Delta[r]/\partial \Delta[r]$ in $\mathbf{sSet}$
and in the homotopy category
$H_{\bullet}$ of pointed simplicial sets and for the corresponding
constant simplicial presheaf in $\M_{\bullet}(S)$.
We write $\GG_{m}$ for the pointed scheme $(\Aff^{1}-0,1)$.

\begin{defn}
\label{BbbKO(A)}
For any pointed motivic space $A$ and any $i$
define
\[
\KO^{[i]}_{r}(A) =
\begin{cases}
Hom_{H_{\bullet}(S)}(A \wedge S^r_{s}, \KO^{[i]}) & \text{for $r \geq 0$},
\\
Hom_{H_{\bullet}(S)}(A \wedge \GG^{\wedge (-r)}_m,\KO^{[i-r]})
& \text{for $r < 0$.}
\end{cases}
\]
\end{defn}

\begin{lem}
\label{NewAndOld}
For any $r$ and $X \in \Sm/S$ one has functorial isomorphisms
$$\alpha_X: KO^{[i]}_r(X) \xra{\cong} \KO^{[i]}_r(X_+).$$
\end{lem}
\begin{proof} In fact, the following chain of isomorphisms give the desired one
$$Hom_{H_{\bullet}}(S^r_{s},KO^{[i]}(X)) \xra{\cong} Hom_{H_{\bullet}}(S^r_{s},\KO^{[i]}(X))
\xra{adj} Hom_{H_{\bullet}(S)}(X_+ \wedge S^r_s ,\KO^{[i]})$$
$$\pi_0(KO^{[i-r]}(X \times \GG^{\times -r}_{m})) \xra{\cong} \pi_0(\KO^{[i-r]}(X \times \GG^{\times -r}_{m})) \xra{adj} Hom_{H_{\bullet}(S)}((X \times \GG^{\times -r}_{m})_+,\KO^{[i-r]})$$
All the arrows are indeed isomorphisms by Lemma
\ref{BbbKOfmotivicfibrant}.
\end{proof}

Now suppose $(X,U) \in \SmOp/S$ with $j \colon U \hra X$ the inclusion and $Z = X -U$.
Schlichting's localization sequence \eqref{E:localization}
\[
KO^{[i]}(X,U) \xrightarrow{e^{*}} KO^{[i]}(X) \xrightarrow{j^{*}} KO^{[i]}(U)
\]
can be described more precisely than we have done so far.
The composite map $j^{*}e^{*}$ is induced by the functor
$Ch^{b}(VBX \ on \ Z) \to Ch^{b}(VBU)$ such that the morphism of functors
$j^{*}e^{*} \to 0$ is a natural weak equivalence.
This natural weak equivalence gives a homotopy from the map of spaces $j^{*}e^{*}$ to the trivial map,
and that gives a factorization
of $e^{*}$ as
\begin{equation}
\label{E:h.fiber}
KO^{[i]}(X,U) \xrightarrow{e_{X,U}} hofib(j^{*}) \xrightarrow{can_{X,U}} KO^{[i]}(X).
\end{equation}
Schlichting's theorem is that $e_{X,U}$ a homotopy equivalence.  This factorization is functorial
in $(X,U)$.

For a pointed motivic space $A$ write $\KO^{[i]}(A) = \fhom(A,\KO^{[i]})$
for the pointed mapping space, which is the
pointed simplicial set $[n] \mapsto Hom_{\M_{\bullet}(S)}(A \wedge \Delta[n]_{+},\KO^{[i]})$.
For $A = X_{+}$ with $X$ a smooth scheme we have
$\KO^{[i]}(A) = \KO^{[i]}(X)$.  For $j \colon U \hra X$ as above we can fill out
a commutative diagram of pointed simplicial sets
\[
\xymatrix{
  KO^{[i]}(X,U) \ar[d]_{e_{X,U}} \\
  hofib(j^{*})\ar[rr]^-{can_{X,U}} \ar [d]_{G_{1}} &&
  KO^{[i]}(X) \ar[r]^-{j^*}  \ar[d]_-{G(X)} & KO^{[i]}(U)  \ar[d]^-{G(U)} \\
  \KO^{[i]}(Cone(j_+))= hofib(j_{+}^{*})  \ar[rr]^-{nat_{X,U}} &&
  \KO^{[i]}(X_{+})   \ar[r]^-{j_{+}^*}  &  \KO^{[i]}(U_{+}) \\
  \KO^{[i]}(X_{+}/U_{+}) \ar[u]^-{q_{X,U}}}
\]

We claim that all the vertical arrows are homotopy equivalences.  This is true for
$G(X)$ and $G(U)$ by Lemma \ref{BbbKOfmotivicfibrant} and therefore for $G_{1}$ because
it is the map between the homotopy fibers.   The map $e_{X,U}$ is a homotopy equivalence
by Schlichting's theorem.  The map $q_{X,U}$ is a homotopy equivalence because it is
obtained by applying $\fhom({-},\KO^{[i]})$ with $\KO^{[i]}$
fibrant to the schemewise weak equivalence between cofibrant objects
$Cone(j_{+}) \to X_{+}/U_{+}$.

The diagram is functorial in $(X,U)$.  We conclude:

\begin{thm}
\label{T:zigzag}
For $(X,U)$ in $\SmOp/S$ there is a functorial zigzag of homotopy equivalences
$KO^{[i]}(X,U) \to \KO^{[i]}(Cone(U_{+} \to X_{+})) \leftarrow \KO^{[i]}(X_{+}/U_{+})$
which
for $(X,\varnothing)$ reduces to the $G(X) \colon KO^{[i]}(X) \to \KO^{[i]}(X)$
of Lemma \ref{BbbKOfmotivicfibrant}.  These induce functorial isomorphisms of groups
$KO^{[i]}_{r}(X,U) \cong \KO^{[i]}_{r}(X_{+}/U_{+})$ for all integers $i$ and $r$.
\end{thm}

\begin{notation}
\label{KOtoBbbKO}
Denote by $\alpha: KO^{[i]}_{*}(-,-) \to \KO^{[i]}_{*}(-_+/-_+)$
the functor isomorphism described in Theorem \ref{T:zigzag}.
\end{notation}

%

\section{The $T$-spectrum $\BO$ and the cohomology theory $\BO^{*,*}$ }
\label{S:BO}

Let $T = \Aff^{1}/(\Aff^{1}-0)$ be the Morel-Voevodsky object.
A \emph{
$T$-spectrum} $E$ over $S$ consists
of a sequence $(E_0,E_1,\ldots)$ of pointed motivic spaces over $S$
plus \emph{structure maps} $\sigma_n\colon E_n \wedge T \to E_{n+1}$.
Let $SH(S)$
denote the stable homotopy category of
$T$-spectra as described in
\cite{Jardine:2000aa}.
It is canonically equivalent
to the motivic stable homotopy category
constructed in
\cite{Voevodsky:1998kx}.

Here we define a $T$-spectrum $\BO$.
Its spaces are $(\KO^{[0]}, \KO^{[1]}, \KO^{[2]}, \KO^{[3]},\dots)$.
We now define the structure maps.

Let
$\Aff^{1} \to \pt$ be the trivial rank one $SL$ bundle, and let
$\thom \in KO^{[1]}_{0}(\Aff^{1},\Aff^{1} -0)$ be its Thom class as defined by \eqref{E:KO.Thom}.
Because $KO_{*}^{[*]}$ is $SL^{c}$ oriented, the maps
\begin{equation}
\label{E:adjoint}
{-} \times \thom \colon KO^{[n]}_{r}(X) \to KO^{[n+1]}_{r}(X \times \Aff^{1}, X \times (\Aff^{1} - 0))
\end{equation}
are isomorphisms.
Recall that $\thom$ is defined by \eqref{E:KO.Thom}
as the class of the symmetric complex
\begin{equation}
\label{E:thom.A1}
\vcenter{
\xymatrix @M=5pt @C=30pt {
K(\OO) \ar[d]^-{\cong}_-{\varTheta(\OO)}
&& 0 \ar[r]
& \OO_{\Aff^{1}} \ar[r]^-{x} \ar[d]_{-1}
& \OO_{\Aff^{1}} \ar[r] \ar[d]^-{1}
& 0
\\
K(\OO)^{\vee}[1]
&& 0 \ar[r]
& \OO_{\Aff^{1}} \ar[r]^-{-x}
& \OO_{\Aff^{1}} \ar[r]
& 0
}}
\end{equation}
of degree $1$ in $Ch^{b}(VB\Aff^{1}\ on \ 0)$.  The maps ${-} \times \thom$ are induced by
the maps of spaces
\begin{equation*}
({-}\boxtimes (K(\OO),\varTheta(\OO)))_{*} \colon
KO^{[n]}(X) \to KO^{[n+1]}(X \times \Aff^{1}, X \times (\Aff^{1} - 0))
\end{equation*}
These maps are thus homotopy equivalences, and
$KO^{[n]} \to KO^{[n+1]}({-} \times \Aff^{1}, {-} \times (\Aff^{1} - 0))$
is a schemewise weak equivalence.

From Lemma \ref{BbbKOfmotivicfibrant} and Theorem \ref{T:zigzag} we now have
a zigzag
\[
\xymatrix @M=5pt @C=20pt {
\KO^{[n]} \ar@{<-}[r]^-{G}_-{\sim}
& KO^{[n]} \ar[r]^-{\times \thom}_-{\sim}
& KO^{[n+1]}({-} \times \Aff^{1}, {-} \times (\Aff^{1} - 0)) \ar[d]^-{\sim}
\\
&& \KO^{[n+1]}(Cone({-} \wedge (\Aff^{1}-0)_{+} \to {-} \wedge \Aff^{1}_{+})) \ar@{<-}[r]^-{\sim}
& \KO^{[n+1]}({-} \wedge T)
}
\]
of schemewise weak equivalences in $\M_{\bullet}(S)$.  Their composition is an isomorphism
$\KO^{[n]} \cong \KO^{[n+1]}({-} \wedge T)$
in the homotopy category.

There is a Quillen adjunction with left adjoint ${-} \wedge T$ and right adjoint
$F({-}) \mapsto F({-} \wedge T)$.   It follows that $\KO^{[n+1]} ({-} \wedge T) $ is
\textbf{fibrant}, while $\KO^{[n]}$ is cofibrant, so there exists a morphism
$\sigma_{n}^{*} \colon \KO^{[n]} \to \KO^{[n+1]} ({-} \wedge T) $ in
$\M_{\bullet}(S)$ representing the same
isomorphism in the homotopy category as the zigzag.
Let $\sigma_{n} \colon \KO^{[n]} \wedge T
\to \KO^{[n+1]}$ be the adjoint morphism.
Since the $KO^{[n]}$ and $\KO^{[n]}$ are periodic modulo $4$, we may choose the
$\sigma_{n}^{*}$ and $\sigma _n$ so they are also periodic.

\begin{defn}
\label{BO}
The $T$-spectrum $\BO$ consists of the sequence of pointed motivic spaces
\begin{equation}
(\KO^{[0]}, \KO^{[1]}, \KO^{[2]}, \KO^{[3]}, \dots)
\end{equation}
together with the structure maps
$\sigma _n:  \KO^{[n]} \wedge T \to \KO^{[n+1]}$ just described.

\end{defn}

The spaces
$\KO^{[n]}$ are motivically fibrant and the adjoints $\sigma_{n}^{*}$ of the structure maps
are schemewise weak equivalences.  So we have \cite[Lemma 2.7]{Jardine:2000aa}:

\begin{thm}
\label{OmegaSpectrumBO}
The $T$-spectrum
$\BO$
is stably motivically fibrant.
\end{thm}


As explained in
\cite{Voevodsky:1998kx}
any $T$-spectrum $E$ defines a bigraded cohomology theory
$(E^{*,*}, \partial)$ on the category $\M_\bullet(S)$ with
\[
E^{p,q}(A) = Hom_{SH_{\bullet}(S)}(A, E \wedge S_{s}^{p-q} \wedge \GG_{m}^{\wedge q} ).
\]
The differential $\partial$ increases the bidegree
by $(1,0)$.

For any $A \in \M_{\bullet}(S)$
the adjunction map induces isomorphisms
\[
\KO^{[n]}_{i}(A) = Hom_{H_{\bullet}(S)}(A \wedge S_{s}^{i}, \KO^{[n]})
\cong Hom_{SH(S)}(A \wedge S_{s}^{i}, \BO \wedge T^{\wedge n}) =
\BO^{2n-i,n}(A),
\]
\begin{multline*}
\KO^{[n]}_{i}(A) =
Hom_{H_{\bullet}(S)}(A \wedge \GG_{m}^{\wedge -i}, \KO^{[n-i]})
\cong Hom_{SH(S)}(A \wedge \GG_{m}^{\wedge -i}, \BO \wedge T^{\wedge n-i})
\\= \BO^{2n-i,n}(A),
\end{multline*}
%
for $i \geq 0$ and $i < 0$ respectively.  This gives us an isomorphism of functors on $\M_{\bullet}(S)$
\[
\beta_A: \KO^{[n]}_{i} \xra{\cong} \BO^{2n-i,n}.
\]

\begin{cor}
\label{KO**andBO**}
The composition isomorphism
$\beta|_{\SmOp/k} \circ \alpha: KO^{[*]}_* \to \BO^{*,*}|_{\SmOp/S}$
respects the boundary homomorphisms in both cohomology theories on $\SmOp/S$.
So it is an isomorphism
$$\gamma: KO^{[*]}_* \to \BO^{*,*}|_{\SmOp/S}$$
of  cohomology theories in the sense of
\cite{Panin:2003rz}.
\end{cor}

\begin{defn}
\label{partialmultonBO**}
Using the cohomology isomorphism $\gamma$ we transplant
to $\BO|_{\SmOp/S}$
the partial multiplicative
structure of $(KO^{[*]}_*, \partial)$ and the Thom and Borel classes of its
$SL^{c}$ orientation described in Theorem \ref{T:SLc.oriented}.
The unit of this partial multiplicative structure is the element
$e= \gamma(\angles{1}) \in \BO^{0,0}(S^0)$.
\end{defn}

\begin{thm}[Bott periodicity]
The adjoints of the structure maps and the categorical periodicity isomorphisms
give levelwise weak equivalences
\[
\BO \xra{\sim \,level} \Omega_{T}^{4}\BO(4)
\xra{\cong} \Omega_{T}^{4}\BO.
\]
\end{thm}

\section{A symplectic version of the Morel-Voevodsky theorem}
\label{S:Morel-Voevodsky}

In
\cite[Theorem 4.3.13]{Morel:1999ab}
Morel and Voevodsky showed that
$\ZZ \times Gr$ represents algebraic $K$-theory in the motivic unstable homotopy category.
If one replaces the ordinary Grassmannians by quaternionic Grassmannians, the same
holds for symplectic $K$-theory.

We write $\HH$ for the trivial rank $2$ symplectic bundle
$\left( \OO^{\oplus 2}, \bigl( \begin{smallmatrix} 0 & 1 \\ -1 & 0 \end{smallmatrix} \bigr) \right)$.
The orthogonal direct sum $\HH^{\oplus n}$ is the trivial symplectic bundle of rank $2n$.
We will sometimes write $\HH^{\oplus n}_{X}$ to designate the trivial symplectic bundle over the
scheme $X$.

The \emph{quaternionic Grassmannian}
$HGr(r,n) = HGr(r,\HH^{\oplus n})$ is defined as the open subscheme of
$Gr(2r,2n) = Gr(2r,\HH^{\oplus n})$ parametrizing subspaces of dimension $2r$ of the fibers of
$\HH^{\oplus n}$
on which the symplectic form of $\HH^{\oplus n}$ is nondegenerate.  We write $\shf U_{r,n}$ for
the restriction to $HGr(r,n)$ of the tautological subbundle of $Gr(2r,2n)$.  The symplectic form of
$\HH^{\oplus n}$ restricts to a symplectic form on $\shf U_{r,n}$ which we
denote by $\phi_{r,n}$.  The pair $(\shf U_{r,n},\phi_{r,n})$ is the
\emph{tautological symplectic subbundle} of rank $2r$ on $HGr(r,n)$.
Morphisms $X \to HGr(r,n)$ are \emph{classified} by subbundles $E \subset \HH_{X}^{\oplus n}$
of rank $2r$ such that the symplectic form of $\HH_{X}^{\oplus n}$
is nondegenerate on every fiber.

More generally, given a symplectic bundle $(E,\phi)$ of rank $2n$ over $X$, the
\emph{quaternionic Grassmannian bundle} $HGr(r,E,\phi)$ is the open subscheme of
the Grassmannian bundle $Gr(2r,E)$ parametrizing subspaces of dimension $2r$
of the fibers of $E$ on which $\phi$
is nondegenerate.

For $r=1$ we have \emph{quaternionic projective spaces} and \emph{bundles}
$HP^{n} = HGr(1,n+1)$ and $HP(E,\phi) = HGr(1,E,\phi)$.

There are commuting morphisms
\begin{equation}
\label{E:alpha.beta}
\vcenter{
\xymatrix @M=5pt @C=40pt {
HGr(r,n) \ar[r]^-{\alpha_{r,n}} \ar[d]_-{\beta_{r,n}}
& HGr(r,n+1) \ar[d]^-{\beta_{r,n+1}}
\\
HGr(r+1,n+1) \ar[r]^-{\alpha_{r+1,n+1}}
&HGr(r+1,n+2)
}}
\end{equation}
with $\alpha_{r,n}$ classified by the rank $2r$ subbundle
$\shf U_{r,n}\oplus 0 \subset \HH^{\oplus n} \oplus \HH$
and $\beta_{r+1,n+1}$ classified by the rank $2r+2$ subbundle
$\HH \oplus \shf U_{r,n}\subset \HH \oplus \HH^{\oplus n}$.  Composition gives us
maps $HGr(n,2n) \to HGr(n+1,2n+2)$.  We define $HGr = \colim HGr(n,2n)$.
We consider $\ZZ \times HGr$ pointed by $(0,HGr(0,0))$.
It has a universal property.

\begin{thm}
\label{T:universal}
Suppose $X \in \Sm/S$ is affine.  Then for every $\xi \in GW^{-}(X)$ there is a morphism
of ind-schemes $f \colon X \to \ZZ \times HGr$
such that $\xi = f^{*}\tau$.  Moreover $f$ is unique up to naive $\Aff^{1}$-homotopy.
\end{thm}

This is the equivalence $\pi_{0}\mathcal L_{6} \cong \pi_{0}\mathrm{GW}^{-}$ of
\cite[Proposition 6.2.1.5]{Barge:2008it} plus the isomorphism $\pi_{0}\mathrm{GW}^{-} =
\mathrm{GW}^{-}$
which happens for our schemes which are regular with $\frac 12$.

For a smooth $S$-scheme $X$  let $KSp(X) = KO(Ch^{b}(VBX),w_{X},
{}^{\vee},-\eta)$ be its symplectic $K$-theory space.
There are natural isomorphisms ${K}Sp \cong {K}O^{[4n+2]}$.

\begin{thm}
\label{T:MV.symp}
The objects $\ZZ \times HGr$ and ${K}Sp$ are isomorphic in the motivic unstable homotopy category
$H_{\bullet}(S)$.
\end{thm}

The proof of this theorem is identical to Morel and Voevodsky's proof for
ordinary $K$-theory except for one point.
Let $\USp(r,n) \to \HGr(r,n)$ be the principal $\Sp_{2r}$-bundle associated to the
tautological rank $2r$ symplectic subbundle on $\HGr(r,n)$.
There is an isomorphism $\USp(r,n) \cong Sp_{2n}/Sp_{2n-2r}$.
To establish Theorem \ref{T:MV.symp} by
following the proof of Morel and Voevodsky exactly we would need  the
$\USp(r,n)$ to form
an \emph{admissible gadget} in the sense of
\cite[Definition 4.2.1]{Morel:1999ab}.
It does not seem as if they do.  Nor can we use the admissible gadget used by
Morel and Voevodsky, for that would substitute for
$\USp(r,n)$ the principal $GL(2r)$-bundle associated to
the tautological subbundle on $Gr(2r,2n)$, called $U_{2r,2n}$ in
\cite{Morel:1999ab}.
(This $U_{2r,2n}$ is
the space of $2r \times 2n$ matrices of maximal rank.)
But the quotient $U_{2r,2n}/Sp_{2r}$ is not the quaternionic Grassmannian $HGr(r,n)$,
and its cohomology risks being much less tractable.

So we give a new definition based on the property actually
used in the proof of
\cite[Proposition 4.2.3]{Morel:1999ab}.
For a commutative ring $B$ let
\begin{align*}
\Delta^{n}_{B} & = \Spec B[t_{1},\dots,t_{n}],
&
\partial \Delta^{n}_{B} & = \Spec B[t_{1},\dots,t_{n}]/(t_{1}t_{2}\cdots t_{n}(1-{\textstyle\sum} t_{i}))
\subset \Delta^{n}_{B}
\end{align*}

\begin{defn}
\label{D:gadget}
An \emph{acceptable gadget} over an $S$-scheme $X$ is a sequence of smooth quasi-projective
$X$-schemes $(U_{i})_{i\in\NN}$ and closed embeddings $U_{i} \to U_{i+1}$ of $X$-schemes
such that for any henselian regular local ring $B$ and any
commutative square
\[
\xymatrix @M=5pt {
\partial \Delta^{n}_{B} \ar[r] \ar[d]_-{\text{inclusion}}
& U_{i} \ar[d]^-{\text{projection}}
\\
\Delta^{n}_{B} \ar[r]
& X
}
\]
there exists a $j \geq i$
and a map $\Delta_{B}^{n} \to U_{j}$ making the following diagram commute.
\[
\xymatrix @M=5pt @C=50pt {
\partial \Delta^{n}_{B} \ar[r] \ar[d]_-{\text{inclusion}}
& U_{i} \ar[r]^-{\text{gadget map}}
& U_{j} \ar[d]^-{\text{projection}}
\\
\Delta^{n}_{B} \ar[rr] \ar[rru]
&& X
}
\]
\end{defn}

Sections of the principal bundle $\USp(r,n) \to HGr(r,n)$ are given by symplectic frames
(i.e.\ symplectic bases)
of the tautological symplectic subbundle.
Therefore giving a morphism $V \to \USp(r,n)$ is equivalent to giving
an embedding $\HH_{V}^{\oplus r} \subset \HH_{V}^{\oplus n}$ such that the symplectic form on
$\HH_{V}^{\oplus r}$ is the restriction of the symplectic form on $\HH_{V}^{\oplus n}$.
This is also equivalent by duality to giving $2n$ sections $s_{1},\dots,s_{2n}$
of the bundle of linear forms $\HH_{V}^{\oplus r \vee}$ such that
$\sum_{i=1}^{n} s_{2i-1} \wedge s_{2i}$ is equal to the symplectic form of $\HH_{V}^{\oplus r}$.

We need the $\USp(r,n)$ to form an acceptable gadget
in the relative case as well (cf.\
\cite[Lemma 4.2.8]{Morel:1999ab}).
Given a symplectic bundle
$(E,\phi)$ of rank $2r$ over $X$, the relative $\USp(E,\phi;n)$ is constructed by taking the
$\Sp_{2r}$-torsor $P \to X$ associated to $(E,\phi)$ and forming the quotient
$(P \times \USp(r,n))/\Sp_{2r}$ by the diagonal action of $\Sp_{2r}$ on the two torsors.
This gives us a fiber bundle $\USp(E,\phi;n) \to X \times \HGr(r,n)$ with fibers
$\Sp_{2r}$ and structural group $\Sp_{2r} \times \Sp_{2r}$
acting on the fibers by left and right translation.
Giving a morphism $V \to \USp(E,\phi;n)$
is equivalent to giving
a triple $(f,g,\iota)$ with $f \colon V \to X$ and $g \colon V \to \HGr(r,n)$ morphisms
of schemes and $\iota \colon f^{*}(E,\phi) \cong g^{*}(\shf {U}_{r,n},\phi_{r,n})$
an isometry
of symplectic bundles.
This is equivalent to giving
$(f,u_{1},\dots,u_{2n})$ with
$f \colon V \to X$ a map and the $u_{i}$ sections
of $f^{*}E^{\vee}$ such
that $\sum_{i=1}^{n} u_{2i-1} \wedge u_{2i} = f^{*}\phi$.

\begin{lem}
\label{L:gadget}
Let $R$ be a commutative ring, let $g\in R$ and let $\overline R = R/(g)$.
Let $(E,\phi)$ a symplectic $R$-module, and let $(\overline E, \overline \phi)$ be the associated
symplectic $\overline R$-module.  Let $u_{1},\dots, u_{2n} \in E^{\vee}$ be linear forms such that
$\sum_{i=1}^{n} \overline u_{2i-1} \wedge \overline u_{2i} =
\overline \phi$ in $\Lambda^{2} \overline{E}^{\vee}$.
Then there
exist linear forms $v_{1},\dots,v_{2n}$ and $w_{1},\dots,w_{2m}$ in $E^{\vee}$ such that
\[
\sum_{i=1}^{n}(u_{2i-1}+gv_{2i-1}) \wedge (u_{2i}+gv_{2i}) + \sum_{j=1}^{m} g w_{2j-1}\wedge g w_{2j} = \phi.
\]
\end{lem}

\begin{proof}
The $\overline u_{1},\dots,\overline u_{2n}$ generate $\overline E^{\vee}$.  So there exist
$v_{1},\dots,v_{2n} \in E^{\vee}$ such that
\[
\phi - \sum_{i=1}^{n} u_{2i-1} \wedge u_{2i} \equiv g \sum_{j=1}^{2n} u_{j} \wedge (-1)^{j}v_{j}
\pmod {g^{2}}.
\qedhere
\]
\end{proof}

\begin{prop}
\label{P:USp}
For any symplectic bundle $(E,\phi)$ of rank $2r$ over a scheme $X$, the
schemes
$(\USp(E,\phi;n))_{n\geq r}$ together with the closed embeddings $\USp(E,\phi;n) \to \USp(E,\phi;n+1)$
induced by the inclusions $\HH^{\oplus n} \subset \HH^{\oplus n} \oplus \HH$
form an acceptable gadget.
\end{prop}

\begin{proof}
Let $R = B[t_{1},\dots,t_{n}]$ and $g = t_{1}t_{2}\cdots t_{n}(1-{\textstyle\sum} t_{i}) \in R$.
Giving the first diagram of Definition \ref{D:gadget} is then equivalent to giving the $(E,\phi)$ and
$(\overline u_{1},\dots, \overline u_{2n})$ of Lemma \ref{L:gadget}.
The map $\Delta_{B}^{n}\to U_{j}$ is then given by
$(u_{1}+gv_{1},\dots,u_{2n}+gv_{2n},gw_{1},\dots,gw_{2m})$.  The top triangle of the second
diagram commutes because modulo $g$ this last vector is
$(\overline u_{1},\dots, \overline u_{2n},0,\dots,0)$.  The lower triangle commutes because of the
equation involving $\phi$.
\end{proof}

Substituting the acceptable gadgets
$\USp(r,n) \to \USp(r,n+1) \to \cdots$ and the quaternionic Grassmannians $HGr(i,n)$
for the admissible gadgets $U_{n,i} \to U_{n,i+1} \to \cdots$ and Grassmannians $Gr(i,n)$
of Morel and Voevodsky,
the proof of
\cite[Theorem 4.3.13]{Morel:1999ab}
can be used to
prove Theorem \ref{T:MV.symp}.
Proposition \ref{P:USp} substitutes for the last paragraph of the proof of Proposition 4.2.3
and for Lemma 4.2.8.
We get $\hocolim_{n} \USp(r,n) \cong \pt$ and
\[
BSp_{2r} \cong B_{et}Sp_{2r} \cong HGr(r,\infty)
\]
in $H(S)$,
and we get the theorem.
At the end one needs the equivalence of hermitian $K$-theories
based on group completion and of Schlichting's Waldhausen-like
construction.  But Schlichting has shown that each is equivalent to the hermitian $Q$-construction:
see
\cite[Theorem 4.2]{Schlichting:2004aa} and \cite[Proposition 6]{Schlichting:2010uq}.

Similar but less definitive results can be proven for the orthogonal group.
Let $\HH_{+}^{\oplus n}$ denote the trivial orthogonal bundle
$(\OO_{S}^{\oplus 2n},q_{2n})$ with the split quadratic form $q_{2n} = \sum_{i=1}^{n}x_{2i-1}x_{2i}$.
Let $RGr(r,2n) = RGr(r,\HH^{\oplus n}_{+})$
be the open subscheme of $Gr(r,2n)$ parametrizing subspaces of rank $r$ on which
$q_{2n}$ is nondegenerate.   Then over $RGr(r,2n)$ we have a tautological rank $r$ orthogonal subbundle
$(U |_{RGr}, q_{2n} |_{RGr})$ whose structural group scheme is the orthogonal group scheme
$O(r,q_{2n}|_{RGr}) \to RGr(r,2n)$.
The associated principal bundle is $RU(r,2n) \to RGr(r,2n)$.  To give a morphism $V \to RU(r,2n)$
one gives a quadratic bundle $(E,q)$ of rank $r$ over $V$ and $2n$ sections $s_{1},\dots,s_{2n}$
of $E^{\vee}$ such that $q = \sum_{i=1}^{n}s_{2i-1}s_{2i}$.
The data $(E,q,s_{1},\dots,s_{2n})$ and $(E_{1},q_{1},t_{1},\dots,t_{2n})$ define the same morphism
if and only if there is an isomorphism $\phi \colon E \cong E_{1}$ such that $q = \phi^{*}q_{1}$
and $s_{i} = \phi^{*}t_{i}$ for all $i$.
The relative case is like the relative symplectic case described earlier.

\begin{prop}
\label{P:RU}
\parens{a}
For any quadratic bundle $(E,q)$ the
$RU(E,q,2n)$ and the inclusions $RU(E,q,2n) \to RU(E,q,2n+2)$ corresponding to
$(E,q,s_{1},\dots,s_{2n}) \mapsto (E,q,s_{1},\dots,s_{2n},0,0)$ form an acceptable gadget, as do their
relative versions.

\parens{b}
For any quadratic bundle $(E,q)$ of rank $r$ over $X$
the fiber bundles $RU(E,q,2n) \to X$ have sections for $n \geq r$ over any open subscheme over which
the vector bundle $E$ trivializes.
\end{prop}

Part (b) of the proposition is of concern
\cite[Proposition 4.1.20 and Definition 4.2.4.3]{Morel:1999ab}.
It is not
an issue for symplectic bundles because symplectic bundles are locally trivial in the Zariski topology.  It
holds because to give
a section of $RU(E,q,2n) \to X$
is to give sections $s_{1},\dots,s_{2n}$ of $E^{\vee}$
such that $q = \sum s_{2i-1}s_{2i}$.  If $E$ trivializes over
$U$ with coordinates $x_{1},\dots,x_{r}$ and $q = \sum_{i \leq j} a_{ij}x_{i}x_{j}$, then one
can give a local section over $U$ by
$s_{2i-1} = x_{i}$ and $s_{2i} = \sum_{j=i}^{n} a_{ij}x_{j}$ for $1 \leq i \leq r$ and $s_{i}=0$ for $i > 2r$.

The results of Proposition \ref{P:RU} and of the arguments of Morel and Voevodsky are isomorphisms
$B_{et}O_{r} \cong RGr(r,\infty)$ and $B_{et}O \cong RGr$
in $H_{\bullet}(S)$.  However, neither of the natural maps
$\ZZ \times BO \to \ZZ \times B_{et}O$ or $\ZZ \times BO \to \KO^{[0]}$
are isomorphisms in $H_{\bullet}(S)$
because orthogonal bundles
are not always locally trivial in the Nisnevich topology.

We also do not know how to calculate $KO_{*}^{[*]}(RGr)$.

\section{The cohomology of quaternionic Grassmannians}
\label{S:symp.Thom}

We review the calculation of the cohomology of quaternionic Grassmannians of \cite{Panin:2010fk}.
We reformulate the definitions and some of the theorems for a bigraded $\veps$-commutative partial multiplication.  In \cite{Panin:2010fk} we assumed we had a full ring structure.
We do not redo the proofs because no change is needed:
all the needed products
are with the Thom and Borel classes of symplectic bundles
or with pullbacks of such classes, and those lie in the $A^{4i,2i}$.

\begin{defn}
[\protect{\cite[Definition 7.1]{Panin:2010fk}}]
\label{Thom}
A \emph{symplectic Thom structure} on a bigraded cohomology theory
$(A^{*,*}, \partial)$ with an $\veps$-commutative partial multiplication or ring structure
is a rule assigning
to every rank $2$ symplectic bundle
$(E, \phi)$
over an $X$ in $\Sm/S$ an element
$\thom(E, \phi) \in A^{4,2}(E,E-X)$
with the following properties:
\begin{enumerate}
\item
For an isomorphism $u: (E, \phi)\cong (E_1, \phi_1)$ one has $\thom(E, \phi) = u^* \thom(E_1, \phi_1)$.
\item
For a morphism $f : Y \to X$ with pullback map $f_E : f^*E \to E$ one has
$f_{E}^* \thom(E, \phi)= \thom(f^*E, f^*\phi)$.
\item
For the trivial rank $2$ bundle $\HH$ over $\pt$
the map
\[
{-} \times \thom(\HH) \colon A^{*,*}(X) \to A^{*, *}(X  \times \Aff^2, X  \times (\Aff^2-\{0\}))
\]
is an isomorphism for all $X$.
\end{enumerate}
The \emph{Borel class} of $(E,\phi)$ is $b(E,\phi) = -z^{*} \, \thom(E,\phi) \in A^{4,2}(X)$ where
$z \colon X \to E$ is the zero section.
\end{defn}

%

From Mayer-Vietoris one sees that for any rank $2$ symplectic bundle
\begin{equation*}
{}\cup \thom(E,\phi) \colon A^{*,*}(X) \xra{\cong} A^{*,*}(E,E-X)
\end{equation*}
is an isomorphism.
The sign in the Borel class is simply conventional.  It is chosen so that if $A^{*,*}$ is an oriented
cohomology theory with an additive formal group law, then the Chern and Borel classes satisfy the
traditional formula $b_{i}(E,\phi) = (-1)^{i}c_{2i}(E)$.
%
%
The following is \cite[Theorem 8.2]{Panin:2010fk}.

\begin{thm}[Quaternionic projective bundle theorem]
\label{QPBTH}
Let $(A^{*,*},\partial)$ be a bigraded cohomology theory with an $\veps$-commutative
partial multiplication or ring structure and
a symplectic Thom structure.  Let
$(\mathcal U, \phi|_{\mathcal U})$ be the tautological rank $2$
symplectic subbundle over
the quaternionic projective bundle $HP(E, \phi)$, and let
$t = b(\mathcal U, \phi|_{\mathcal U})$
be its Borel class. Then
we have an isomorphism of graded $A^{0}(X)$-modules
$$(1, t, \dots , t^{n-1}) \colon
 A^{*,*}(X) \oplus A^{*,*}(X) \oplus \dots \oplus
A^{*,*}(X) \to A^{*,*}(HP_X(E, \phi))$$
and an isomorphism
of graded modules over $A^{0}(X) = \bigoplus_{p}A^{2p,p}(X)$\textup{:}
$$(1, t, \dots , t^{n-1}) \colon
 A^{0}(X) \oplus A^{0}(X) \oplus \dots \oplus A^{0}(X) \to A^{0}(HP_X(E, \phi)).$$
\end{thm}

\begin{defn}
\label{PontriaginClasses}
Under the hypotheses of Theorem
\ref{QPBTH}
there are unique elements
$b_i(E, \phi) \in A^{4i,2i}(X)$ for $i=1,2, \dots , n$
such that
$$t^{n}-b_1(E, \phi)\cup t^{{n-1}} + b_2(E, \phi)\cup t^{{n-2}} - \dots + (-1)^n b_n(E, \phi)=0.$$
The classes
$b_i(E, \phi)$
are called the \emph{Borel classes}
of $(E, \phi)$ with respect to the symplectic Thom structure of the cohomology theory $(A, \partial)$.
For $i > n$ and $i < 0$ one sets $b_i(E, \phi)$ = 0, and one sets $b_0(E, \phi) = 1$.
\end{defn}


The quaternionic projective bundle theorem has consequences the symplectic splitting principle and
the Cartan sum formula for Borel classes \cite[Theorems 10.2, 10.5]{Panin:2010fk}.
With them one can compute the cohomology of quaternionic Grassmannians.
Let $e_{i}$ denote the $i^{\text{th}}$ elementary symmetric polynomial, and let $h_{i}$ be the
$i^{\text{th}}$ complete symmetric polynomial, the sum of all monomials of degree $i$.
There are formulas relating them, including the recurrence relation
$h_{k} + \sum_{i\geq 1} (-1)^{i}e_{i}h_{k-i} = 0$.  Let $\Pi_{r,n-r}$ be the set of all partitions whose
Young diagrams fit inside an $r \times (n-r)$ box.  (More formally these are partitions
$\lambda$ with $l(\lambda) =\lambda'_{1}\leq r$ and $\lambda_{1} \leq n-r$.)
Associated to any such partition is a
\emph{Schur polynomial}, which can be written in terms of the $e_{i}$.

\begin{thm}
\label{T:Grass}
For any bigraded ring
cohomology theory $A^{*,*}$ with an $\varepsilon$-commutative partial multiplication or ring structure and
a symplectic Thom structure and any $X$ the map
\begin{equation}
\label{E:HGr.cohom.1}
A^{*,*}(X)[e_{1},\dots,e_{r}] / (h_{n-r+1},\dots,h_{n}) \xrightarrow{\cong} A^{*,*}(\HGr(r,n) \times X)
\end{equation}
sending $e_{i} \mapsto b_{i}(\shf U_{r,n},\phi_{r,n})$ for all $i$
is an isomorphism, as is the map
\begin{equation}
\label{E:HGr.cohom.2}
A^{*,*}(X)^{\oplus \binom{n}{r}} \xrightarrow{(s_{\lambda}(\shf U_{r,n},\phi_{r,n}))_{\lambda\in\Pi_{r,n-r}}}
A^{*,*}(\HGr(r,n) \times X).
\end{equation}
%
\end{thm}

\begin{thm}
[\protect{\cite[Theorem 11.4]{Panin:2010fk}}]
\label{T:HGr.lim.cohom}

For any bigraded ring
cohomology theory $A^{*,*}$ with an $\varepsilon$-commutative partial multiplication or ring structure and
a symplectic Thom structure and any $X$ the $\alpha_{r,n}$ and
$\beta_{r,n}$ of \eqref{E:alpha.beta} induces split surjections
\begin{gather*}
(\alpha_{r,n} \times 1_{X})^{*} \colon A(\HGr(r,n+1) \times X) \onto A(\HGr(r,n) \times X)
\\
(\beta_{r,n} \times 1_{X})^{*} \colon A(\HGr(r+1,n+1) \times X) \onto A(\HGr(r,n) \times X)
\end{gather*}
which the isomorphisms \eqref {E:HGr.cohom.2} identify with the surjections
$A(X)^{\oplus \binom{n+1}{r}} \onto A(X)^{\oplus \binom{n}{r}}$ and
$A(X)^{\oplus \binom{n+1}{r+1}} \onto A(X)^{\oplus \binom{n}{r}}$
which are the identity on the summands
corresponding to $\lambda \in \Pi_{n,r}$ and
which vanish on the summands
corresponding to $\lambda \not \in \Pi_{n,r}$.
We have isomorphisms
\begin{gather}
\label{E:HGr(r,infty)}
A^{*,*}(X) [[b_{1},\dots,b_{r}]]^{\homog} \xrightarrow{\cong}
\varprojlim\limits_{n \to \infty} A^{*,*}(\HGr(r,n) \times X)
\\
\label{E:HGr(infty,2.infty)}
A^{*,*}(X)[[b_{1},b_{2},  b_{3}, \dots ]]^{\homog}
\xrightarrow{\cong}
\varprojlim
\limits_{n\to \infty}
A^{*,*}(\HGr(n,2n) \times X)
\end{gather}
with each variable $b_{i}$ sent to the inverse system of $i^{\text{th}}$ Borel classes
$(b_{i}(\shf U_{r,n}))_{n \geq r}$ or
$(b_{i}(\shf U_{n,2n}))_{n\in \NN}$.
\end{thm}

The notation on the left in \eqref {E:HGr(r,infty)}--\eqref {E:HGr(infty,2.infty)} is
the bigraded ring of homogeneous power series.
We have a simple lemma.

\begin{lem}
\label{BO(AwedgeB)andBO(AtimesB)}
If $A$ is a $T$-spectrum, then for any pointed motivic spaces $X$ and $Y$
the canonical map $X \times Y \to X\wedge Y$
induces a split injection
$A^{r,s}(X \wedge Y) \hra A^{r,s}(X \times Y)$.
The image of the injection coincides with the kernel of the map
$$[(id_X \times y)^*, (x \times id_Y^*)] \colon A^{r,s}(X  \times Y) \to
A^{r,s}(X  \times y) \oplus A^{r,s}(x  \times Y).$$
\end{lem}

We write $[-n,n] = \{ i \in \ZZ \mid -n \leq i \leq n\}$.  We have a sequential colimit of pointed spaces
\[
(\ZZ \times HGr, (0,x_{0})) = \colim ( ([-n,n] \times HGr(n,2n), (0,x_{0}))
\]
to which Theorem \ref {T:varprojlim.a} applies.  Theorem \ref{T:HGr.lim.cohom} and
Lemma \ref{BO(AwedgeB)andBO(AtimesB)} now give the following result.

\begin{thm}
\label{T:A(ZxHGr)}
Let $A$ be a $T$-spectrum whose associated cohomology theory $(A^{*,*},\partial)$ has an
$\varepsilon$-commutative partial multiplication or ring structure and a symplectic Thom structure.
Then the natural map
\[
A^{*,*}((\ZZ \times HGr),(0,x_{0})) \to \varprojlim A^{*,*}([-n,n] \times HGr(n,2n),(0,x_{0}))
\]
is an isomorphism.  More generally for any $r$ and $s$ the natural map
\[
A^{*,*}((\ZZ \times HGr,(0,x_{0}))^{\wedge r} \wedge (HP^{1},x_{0})^{\wedge s})
\to \varprojlim A^{*,*}(([-n,n] \times HGr(n,2n),(0,x_{0}))^{\wedge r}\wedge (HP^{1},x_{0})^{\wedge s})
\]
is an isomorphism.
\end{thm}

We complete our review of parts of \cite{Panin:2010fk} by looking at the geometry of
$HP^{1} = HGr(1,\HH^{\oplus 2})$.

\begin{thm}
\label{T:HP1.T2}
In $H_{\bullet}(S)$ we have a canonical isomorphism
$\eta \colon (HP^{1}_{+},\pt) \cong T^{\wedge 2}$.
\end{thm}

\begin{proof}

By definition $HP^{1}$ is by the open subscheme of $Gr(2,4)$ parametrizing
$2$-dimensional symplectic subspaces of the $4$-dimensional trivial symplectic bundle.
It contains two distinguished points
$x_{0} = [\HH \oplus 0 ]$ and
$x_{\infty} = [0 \oplus \HH]$.
The $x_{\infty}$ is the origin of an open cell $\Aff^{4} \subset Gr(2,4)$ of the usual Grassmannian
consisting of subspaces with basis of the form $(y_{1},y_{2},1,0)$ and $(y_{3},y_{4},0,1)$.
Within the $\Aff^{4}$ there are two transversal loci $N^{+} \cong \Aff^{2}$ defined by $y_{2}=y_{4}=0$
and $N^{-} \cong \Aff^{2}$ defined by $y_{1} = y_{3}= 0$.  They are closed in $HP^{1}$.

%
The complement $HP^{1}-N^{+}$ is the quotient of $\Aff^{5}$ by a free action of $\GG_{a}$
\cite[Theorem 3.4]{Panin:2010fk}.
Consequently the structural map $HP^{1}-N^{+} \to \pt$ and its section
$x_{0} \colon \pt \to HP^{1}-N^{+}$ are motivic weak equivalences.
%
This gives us motivic weak equivalences
\begin{equation}
\label{E:T2=HP1}
\vcenter{
\xymatrix @M=5pt @C=17.5pt {
T^{\wedge 2} \cong N^{-}/(N^{-}-0) \ar[d]_-{\Aff}
\ar[r]^-{\text{2 out of 3}} \ar[rd]|-{\text{2 out of 3}}
&
HP^{1}/(HP^{1}-N^{+})
&
(HP^{1},x_{0}) \ar[l]_-{/\Aff}
\\
\Aff^{4}/(\Aff^{4}-N^{+}) \ar@{<-}[r]_-{\text{excision}}
&
(\Aff^{4}\cap HP^{1})/((\Aff^{4}\cap HP^{1})-N^{+}) \ar[u]_-{\text{excision}}
&
}}
\end{equation}
The zigzag on the top line is the canonical isomorphism in $H_{\bullet}(S)$.
\end{proof}


A symplectic bundle $(E,\phi)$ is naturally a special linear bundle $(E,\lambda_{\phi})$ with
$\lambda_{\phi}$ the inverse of the Pfaffian of $\phi$.
Hence special linear Thom classes of hermitian $K$-theory (Theorem \ref{T:SLc.oriented})
give $\BO^{*,*}$ a symplectic Thom structure.
So there are Borel classes for symplectic bundles in hermitian $K$-theory,
and all the formulas of this section are valid for them.

We may compute the Borel classes induced by the Thom classes of this particular
symplectic Thom structure.  We need the isomorphism of \eqref {E:KSp}, which becomes the isomorphism
%
\begin{equation}
\label{E:GW-.isom}
\begin{array}{ccc}
GW^{-}(X) & \overset{\cong}{\lra} & KO_{0}^{[2]}(X)
\\
{}[X,\phi] & \longmapsto & -\bigl[ (X,\phi)[1] \bigr].
\end{array}
\end{equation}
The sign makes the isomorphism commute
with the forgetful maps to $K_{0}(X)$.
For a rank $2$ symplectic bundle $(E,\phi)$
has Borel class $b_{1}(E,\phi) = -z^{*}\thom(E,\OO_{X},\lambda_{\phi})$, which is the image under
the isomorphism of $[E,\phi]-[\HH] \in GW^{-}(X)$.  The symplectic splitting principle
\cite[Theorem 10.2]{Panin:2010fk}
now gives the following.

\begin{prop}
\label{P:p1.formula}
Let $(F,\psi)$ be a symplectic bundle of rank $2r$ on $X$.  Its first Borel class
$b_{1}(F,\psi) \in KO_{0}^{[2]}(X)$
is the image under the isomorphism \eqref{E:GW-.isom} of $[F,\psi]-r[\HH] \in GW^{-}(X)$.
\end{prop}

Formulas for higher Borel classes in terms of exterior powers of $(F,\psi)$ will
be given in \cite{Walter:2010ab}.

\section{The strategy for putting a ring structure on $A^{*,*}$ }

We explain our strategy for turning our partial multiplicative structure on $\BO^{*,*}$
into a full ring structure.

We will need the following standard facts about spectra.
Recall that a motivic space $X$ is \emph{small} if $Hom_{SH(S)}(\Sigma_{T}^{\infty}X,{-})$
commutes with arbitrary coproducts.  We write $\M^{\ft}_{\bullet}(S)$ for the full subcategory of
small motivic spaces.

\begin{thm}
[\protect{\cite[Lemma A.34]{Panin:2009aa}}]
\label{T:varprojlim.a}
Let $D^{(0)} \to  D^{(1)} \to D^{(2)} \to \cdots$ be a sequence of morphisms in $SH(S)$ with
$\hocolim D^{(i)} = D$, let $X$ be a small motivic space, and let $A$ be a $T$-spectrum.  Then
there is a canonical isomorphism
\[
Hom_{SH(S)}(\Sigma_{T}^{\infty}X,D) = \colim Hom_{H_{\bullet}(S)}(X,D^{(n)}).
\]
and a canonical short exact sequence:
\[
0 \to \varprojlim\nolimits^{1} A^{p-1,q}(D^{(i)}) \to A^{p,q}(D) \to
\varprojlim A^{p,q}(D^{(i)}) \to 0.
\]
\end{thm}

Particular cases of this are the following.

\begin{thm}
[\protect{\cite[Theorem 5.2]{Panin:2010fk}}]
\label{T:small.colim}
Let $X$ be a small motivic space and $A$ a $T$-spectrum.  Then we have
\[
Hom_{SH(S)}(\Sigma_{T}^{\infty}X,A) = \colim Hom_{H_{\bullet}(S)}(X \wedge T^{\wedge n}, A_{n}).
\]
\end{thm}

\begin{thm}
[\protect{\cite[Corollaries 3.4, 3.5]{Panin:2009aa}}]
\label{T:varprojlim.b}
Let $A$ and $E$ be $T$-spectra.  Then for any $r$
there is a canonical short exact sequences
\begin{gather*}
0 \to \varprojlim\nolimits^{1} A^{*-2rn-1,*-rn}(E_{n}^{\wedge r}) \to A^{*,*}(E^{\wedge r}) \to
\varprojlim A^{*-2rn,*-rn}(E_{n}^{\wedge r}) \to 0.
\end{gather*}
\end{thm}

\begin{defn}
\label{D:almost.monoid}
An \emph{almost commutative monoid} in $SH(S)$ is a triple $(A,\mu,e)$ with $A$ a $T$-spectrum and
$\mu \colon A \wedge A \to A$ and $e \colon \Sigma_{T}^{\infty}\boldsymbol{1}\to A$ morphisms
in $SH(S)$ such that
\begin{enumerate}
\item the morphism $\mu \circ (\mu \wedge 1) - \mu \circ (1 \wedge \mu) \in Hom_{SH(S)}(
A \wedge A \wedge A, A)$ lies in the subgroup
$\varprojlim^{1} A^{*-6n-1,*-3n}(A_{n}\wedge A_{n}\wedge A_{n})$,
\item for $\sigma \colon A \wedge A \to A \wedge A$ the morphism switching the two factors,
the morphism $\mu - \mu \circ \sigma \in \Hom_{SH(S)}( A \wedge A,A)$
lies in the subgroup
$\varprojlim^{1} A^{*-4n-1,*-2n}(A_{n}\wedge A_{n})$,
\item the map $1 - \mu \circ (1 \wedge e) \in \Hom_{SH(S)}(A,A)$ lies in the subgroup
$\varprojlim^{1} A^{*-2n-1,*-n}( A_{n})$.
\end{enumerate}
\end{defn}

An almost commutative monoid $(A,\mu,e)$ defines pairings
\begin{equation}
\label{E:m.pairing.1}
\times \colon A^{p,q}(X) \times A^{r,s}(Y) \to
A^{p+r,q+s}(X \wedge Y)
\end{equation}
for $X$ and $Y$ in $\M_{\bullet}(S)$ as follows.  For
$\alpha: \Sigma_{T}^{\infty}X \to A \wedge S^{p,q} $ and $\beta: \Sigma_{T}^{\infty}Y \to A \wedge S^{r,s} $ define
$\alpha \times \beta \in A^{p+r,q+s}(X \wedge Y)$
as the composition
\begin{equation*}
\Sigma_{T}^{\infty}(X \wedge Y) \cong \Sigma_{T}^{\infty}X \wedge \Sigma_{T}^{\infty}Y \xra{\alpha \wedge \beta}
A \wedge S^{p,q} \wedge A \wedge S^{r,s} \cong A \wedge A \wedge S^{p+r,q+s}
\xra{m \wedge 1 }A  \wedge S^{p+r,q+s}.
\end{equation*}
The unit $e\in Hom_{SH(S)}(\Sigma_{T}^{\infty}\pt_{+},A)$ defines an element $1\in A^{0,0}(\pt_{+})$.
There is then a unique element $\veps \in A^{0,0}(\pt_{+})$ such that
$\Sigma_{T}\veps \in Hom_{SH(S)}(T, A\wedge T)$ is the composition of the endomorphism
of $T$ induced by the endomorphism $x \mapsto -x$ of $\Aff^{1}$ with $e \wedge 1_{T}$.

\begin{thm}
\label{T:almost.commutative.product}
For an almost commutative monoid $(A,\mu,e)$ in $SH(S)$ the cohomology theory $(A^{*,*},\partial)$
with the pairing $\times$ of \eqref{E:m.pairing.1} and the element $1 \in A^{0,0}(\pt_{+})$ form
an $\veps$-commutative ring cohomology theory on $\M_{\bullet}(S)^{\ft}$ and on $\SmOp/S$.
\end{thm}

\begin{proof}
There are canonical elements $a_{n} \colon \Sigma_{T}^{\infty}A_{n}(-n) \to A$.  The definition of an
almost commutative monoid is equivalent to having $(A,\mu,e)$ such that the induced pairing satisfies
$(a_{n} \times a_{n}) \times a_{n} = a_{n} \times (a_{n} \times a_{n})$ and
$\sigma^{*}(a_{n} \times a_{n}) = \varepsilon^{n} a_{n} \times a_{n}$ and $a_{n} \times e = a_{n}$
for all $n$.  It then also satisfies
\[
(\Sigma^{p,q}a_{n} \times \Sigma^{p',q'}a_{n}) \times \Sigma^{p'',q''}a_{n} =
\Sigma^{p,q}a_{n} \times (\Sigma^{p',q'}a_{n} \times \Sigma^{p'',q''}a_{n})
\]
for all $(p,q)$, $(p',q')$ and $(p'',q'')$.
By Theorem \ref{T:small.colim} for a small motivic space $X$ any morphism
$\Sigma_{T}^{\infty}X \to A\wedge S^{p,q}$ factors as
\[
\Sigma_{T}^{\infty}X \to \Sigma_{T}^{\infty}A_{n}(-n) \wedge S^{p,q}
\xra{\Sigma^{p,q}a_{n}} A \wedge S^{p,q}
\]
for some $n$.  So on small motivic spaces the multiplication is associative.
The $\varepsilon$-commutativity and the unit property are treated similarly.
The signs in the commutativity come from permuting the spheres.
\end{proof}

\section{The universal elements}
\label{S:universal}

The isomorphism $\tau \colon (\ZZ \times HGr, (0,x_{0})) \xra{\cong} \KSp$ of Theorem
\ref{T:MV.symp} is classified by an element which has restrictions
%
%
%
\begin{equation}
\label{E:tau.formula}
\tau |_{\{i\} \times HGr(n,2n)}
= [\mathcal U_{n,2n},\phi_{n,2n}]+(i-n)[\HH] \in KSp_{0}(HGr(n,2n)).
\end{equation}
Its image under the isomorphism \eqref {E:GW-.isom} is a class
$\tau_{2} \in \KO^{[2]}_{0}(\ZZ \times HGr,(0,x_{0}))$ with
\begin{equation*}
\tau_{2} |_{\{i\} \times HGr(n,2n)}
= b_{1}(\mathcal U_{n,2n},\phi_{n,2n})+i\hh\in \KO_{0}^{[2]}(HGr(n,2n))
\end{equation*}
according to Proposition \ref{P:p1.formula}.
Here $\hh \in \KO^{[2]}_{0}(\pt) = \BO^{4,2}(\pt)$ is the class corresponding to $[\HH] \in GW^{-}(\pt)$
under the isomorphism.  By Theorem \ref{T:HGr.lim.cohom} we have an isomorphism
\[
\BO^{*,*}(\ZZ \times \HGr)
\cong \prod_{i \in \ZZ}\BO^{*,*}(\pt)[[b_{1},b_{2},  b_{3}, \dots ]]^{\homog}
\]
with the product taken in the category of graded rings.  Setting
\begin{align*}
b_{1} = (b_{1})_{i \in \ZZ} \in \BO^{*,*}(\ZZ \times HGr),
&&
\tfrac 12 \rk = (i1_{\BO})_{i \in \ZZ}\in \BO^{*,*}(\ZZ \times HGr)
\end{align*}
we see we have $\tau_{2} = b_{1} + (\frac 12 \rk) \hh$.

For any $k$ we have a composition of isomorphisms
in the homotopy category $H_{\bullet}(S)$
\[
(\ZZ \times HGr, (0,x_{0})) \xra{\tau} \KSp \xrightarrow{trans_{2k+1}}
\KO^{[4k+2]}
\xrightarrow{-1} \KO^{[4k+2]}
\]
where the $trans_{2k+1}$ comes from translation
and the $-1$
is the inverse operation of the $H$-space structure, which we add as in \eqref{E:KSp}
so that the forgetful maps to $K$-theory should commute up to homotopy.
(The inverse in the $H$-space structure comes from the
$\Omega_{T}$-spectrum structure and to the authors' knowledge not
from a duality-preserving functor.)

\begin{defn}[Universal element]
\label{D:tau}
We denote by
\begin{equation*}
\tau_{4k+2} \in \KO_{0}^{[4k+2]}(\ZZ \times HGr, (0,x_{0}))
\cong \BO^{8k+4,4k+2}(\ZZ \times HGr, (0,x_{0}))
\end{equation*}
the element corresponding to the composition.  It is given by
%
$\tau_{4k+2} = (b_{1}+ (\tfrac 12 \rk)\hh) \cup \beta_{8}^{k}$.
%
\end{defn}

%
%
%

%
%

We write $[-n,n] = \{ i \in \ZZ \mid -n \leq i\leq n\}$.  The class
$\tau \in GW^{-}(\ZZ \times HGr)$ giving the isomorphism $\ZZ \times HGr \cong \KSp$
has a universal property.


%

\begin{lem}
\label{L:mu}
There are unique $\mu$ and $\mu_{8k+4}$ in $H_{\bullet}(S)$ making the diagram commute
\[
\xymatrix @M=5pt @C=40pt {
(\ZZ \times HGr,(0,x_{0})) \wedge (\ZZ \times HGr,(0,x_{0}))
\ar[r]^-{\mu}
\ar[d]_-{\tau_{4k+2}\wedge \tau_{4k+2}}^-{\cong}
& \KO^{[0]} \ar[d]^-{\text{\textup{translation}}} _-{\cong}
\\
\KO^{[4k+2]}\wedge \KO^{[4k+2]}
\ar[r]_-{\mu_{8k+4}}
&
\KO^{[8k+4]}
}
\]
and such that for each $i$, $j$ and $n$ the restriction of $\mu$
\[
(\{i\} \times HGr(n,2n)) \times (\{j\} \times HGr(n,2n)
\to \KO^{[0]}
\]
is the class in $H_{\bullet}(S)$ of the morphisms of ind-schemes
represing the orthogonal Grothendieck-Witt class which is
\[
([\shf U_{n,2n},\phi_{n,2n}]+(i-n)[\HH]) \boxtimes ([\shf U_{n,2n},\phi_{n,2n}]+(j-n)[\HH]).
\]
\end{lem}

The $\mu$ with the asserted restrictions exists and is unique because of Theorem \ref {T:A(ZxHGr)}.
It factors through the wedge space because of Lemma \ref {BO(AwedgeB)andBO(AtimesB)}.
Then $\mu_{8k+4}$ is the map making the diagram commute.

\begin{lem}
\label{L:mu.compatible}
The following diagram commutes in $H_{\bullet}(S)$.
\[
\xymatrix @M=5pt @C=40pt {
\KO^{[4k-2]}\wedge \KO^{[4k-2]} \wedge T^{\wedge 8}
\ar[r]^-{\mu_{8k-4} \wedge 1}
\ar[d]_-{\tau_{4k+2}\wedge \tau_{4k+2}}^-{\cong}
& \KO^{[8k-4]} \wedge T^{\wedge 8}\ar[dd]^-{\text{\textup{structure maps}}}
\\
\KO^{[4k-2]}\wedge T^{\wedge 4}\wedge \KO^{[4k-2]} \wedge T^{\wedge 4}
\ar[d]^{\text{\textup{structure maps}}}
\\
\KO^{[4k+2]}\wedge \KO^{[4k+2]}
\ar[r]_-{\mu_{8k+4}}
&
\KO^{[8k+4]}
}
\]
\end{lem}

\begin{proof}
Because of Theorem \ref {T:A(ZxHGr)} it is enough to observe that the compositions
of the two paths of arrows with the composition
\[
\xymatrix @M=5pt @C=20pt {
(\{i\} \times HGr(n,2n)) \times (\{j\} \times HGr(n,2n)) \times (HP^{1},x_{0})^{\times 4}
\ar@{^{(}->}[d]
\\
(\ZZ \times HGr) \times (\ZZ \times HGr) \times (HP^{1},x_{0})^{\times 4}
\ar[r]^-{\tau_{4k-2}\wedge \tau_{4k-2}\wedge \eta^{\wedge 4}} _-{\cong}
& \KO^{[4k-2]}\wedge \KO^{[4k-2]} \wedge T^{\wedge 8}
}
\]
are the same for all $i$, $j$ and $n$.
\end{proof}

By Theorem \ref {T:varprojlim.b}
we have a surjection
\[
Hom_{SH(S)}(\BO \wedge \BO, \BO) \to
\varprojlim\BO^{16k+8,8k+4}(\KO^{[4k+2]}\wedge \KO^{[4k+2]}) \to 0
\]
the compositions
\[
\Sigma_{T}^{\infty}(\KO^{[4k+2]}\wedge \KO^{[4k+2]})(-8k-4) \xrightarrow{\mu_{4k+2}} \Sigma_{T}^{\infty}\KO^{[8k+4]}(-8k-4) \xrightarrow{\text{\textup{canonical}}} \BO
\]
form a system of elements of the groups in the inverse limit.  They are compatible with the connecting
maps of the inverse limit because the diagrams of Lemma \ref {L:mu.compatible} commute.
Let
\begin{equation}
\label{E:bar.m}
\bar m = (\mu_{8k+4})_{k \in \NN}\in \varprojlim\BO^{16k+8,8k+4}(\KO^{[4k+2]}\wedge \KO^{[4k+2]})
\end{equation}
and let
\begin{equation}
\label{E:m}
m \in Hom_{SH(S)}(\BO\wedge \BO)
\end{equation}
be an element lifting it.  As in \eqref {E:m.pairing.1} $m$ defines a pairing
\begin{equation}
\label{E:m.pairing}
\times \colon \BO^{p,q}(A) \times \BO^{r,s}(B) \to
\BO^{p+r,q+s}(A \wedge B).
\end{equation}
 %


\begin{thm}
\label{T:coincide}
Suppose $S$ satisfies the hypotheses of Theorem \ref{T:lim1}.
For $X$ and $Y$ in $\Sm/S$ and all $p$ and $q$ the pairing
\begin{align*}
\BO^{4p,2p}(X_{+}) \times \BO^{4q,2q}(Y_{+}) &
\to \BO^{4p+4q,2p+2q}(X_{+} \wedge Y_{+})
\\
\intertext{induced by $m$ is identified via the isomorphism $\gamma$ of  Corollary \ref{KO**andBO**}
with the naive pairing}
KO_{0}^{[2p]}(X) \times KO_{0}^{[2q]}(Y) & \to KO_{0}^{[2p+2q]}(X \times Y).
\end{align*}
\end{thm}

\begin{proof}
Because of Jouanolou's trick it is enough to consider affine $X$ and $Y$.

Let $\alpha \in KO_{0}^{[2p]}(X)$ and $\beta \in KO_{0}^{[2q]}(Y)$.

When we have
$2p = 2q = 4k+2$ and $X = Y = [-n,n] \times HGr(n,2n)$ and
$\alpha = \beta$ is the restriction of the universal class $\tau_{4k+2}$, then we have the identification
by the construction of $m$.

When we have $2p=2q=4k+2$ but $X$, $Y$, $\alpha$ and $\beta$ are general then by Theorem
\ref{T:universal} there exists an $n$ and morphisms $f_{\alpha} \colon X \to [-n,n] \times HGr(n,2n)$
and $f_{\beta} \colon Y \to [-n,n] \times HGr(n,2n)$ in $\Sm/S$
such that
$f_{\alpha}^{*}(\tau_{4k+2} |_{[-n,n] \times HGr(n,2n)}) = \alpha$
and
$f_{\beta}^{*}(\tau_{4k+2} |_{[-n,n] \times HGr(n,2n)}) = \beta$.  The
identification of the two pairings for $\alpha$ and $\beta$ now follows from that for the restrictions
of the universal classes because both pairings are functorial for morphisms in $\Sm/S$.

For general $p$ and $q$ pick $4k+2 \geq \max(2p,2q)$.  The identification of the two pairings on
$\alpha$ and $\beta$ follows from the identification of the two pairings on
\begin{gather*}
\alpha \times b_{1}(U,\phi)^{\times 2k+1-p} \in KO_{0}^{[4k+2]}(X \times (HP^{1})^{\times 2k+1-p}),
\\
\beta \times b_{1}(U,\phi)^{\times 2k+1-q} \in KO_{0}^{[4k+2]}(Y \times(HP^{1})^{\times 2k+1-q}).
\qedhere
\end{gather*}
\end{proof}



\begin{thm}
\label{monoidBO}
Let $m$ be as in \eqref {E:m}, and
let $e \in \BO^{0,0}(\pt_{+})$
be the unit of the partial multiplicative structure on
$(\BO^{*,*},\partial)$
of Definition \ref{partialmultonBO**}.
Then
\( (\BO, m, e)\)
is an almost commutative monoid in  $SH(S)$, and the $\times$ and $1$ induced by $m$ and $e$
make $(\BO^{*,*},\partial,\times,1)$ a $\angles {-1}$-commutative bigraded ring cohomology theory on
$\M_{\bullet}^{\ft}(S)$ and on $\SmOp/S$.
\end{thm}

\begin{proof}
We prove associativity.  The morphisms
\[
\xymatrix @M=5pt @C=60pt {
\BO \wedge \BO \wedge \BO
\ar@<4pt>[r]^-{m \circ (m \wedge id_{\BO} )}
\ar@<-4pt>[r]_-{m \circ (id_{\BO} \wedge m)}
&
\BO
}
\]
define two elements of $\BO^{0,0}(\BO \wedge \BO \wedge \BO)$
with images in
\[
\varprojlim \BO^{24k+12,12k+6}(\KO^{[4k+2]} \wedge \KO^{[4k+2]} \wedge
\KO^{[4k+2]}).
\]
This last group is isomorphic to
\[
\varprojlim \BO^{24k+12,12k+6}((\ZZ \times HGr,(0,x_{0}))^{\wedge 3})
\]
So it is enough to show that the corresponding elements
$(\tau_{4k+2} \boxtimes \tau_{4k+2}) \boxtimes \tau_{4k+2}$ and
$\tau_{4k+2} \boxtimes (\tau_{4k+2} \boxtimes \tau_{4k+2})$
in each group of the inverse system are
equal.  However, since we have
\[
\BO^{24k+12,12k+6}((\ZZ \times HGr,(0,x_{0}))^{\wedge 3})
\cong
\varprojlim \BO^{24k+12,12k+6}(([-n,n] \times HGr(n,2n),(0,x_{0}))^{\wedge 3})
\]
by Theorem \ref{T:A(ZxHGr)}  it is enough to show that the restrictions to the
smooth affine schemes
$[-n,n] \times HGr(n,2n)$ are equal.  But then by Theorem \ref{T:coincide} the
pairings coincide with the naive products, and those are associative.

The proofs of commutativity and of the unit property are similar.

Thus $(\BO,m,e)$ is an almost commutative monoid in $SH(S)$.  The rest of the statement
of the theorem follows from that fact by Theorem \ref{T:almost.commutative.product}
except for the value of $\veps$.  For that note that $\Sigma_{T}1 \in \BO^{2,1}(T)$ is
the class in $GW^{[1]}(\Aff^{1},\Aff^{1}-0)$ of the symmetric complex $(K(\OO),\varTheta(\OO))$
of \eqref{E:thom.A1}.  The pullback of the complex along the endomorphism $x \mapsto -x$
of $\Aff^{1}$ is isometric to $(K(\OO),-\varTheta(\OO))$.  So we have $\veps = \angles{-1}$.
%
\end{proof}

\begin{thm}
\label{T:unique.prod}
Suppose $\times$ and $\times'$ are two products on $(\BO^{*,*},\partial)$ on $\M_{\bullet}^{\ft}(S)$
which associative and $\angles{-1}$-commutative with unit $1 = \angles 1$, and such that
$\alpha \times \Sigma_{\GG_{m}}1 = \Sigma_{\GG_{m}}\alpha$ and $\alpha \times \Sigma_{S^{1}_{s}}1
= \Sigma_{S^{1}_{s}}\alpha$
for all $\alpha$.  If $\times$ and $\times'$
are both compatible
with the products $KO^{[2m]}_{0}(X) \times KO^{[2n]}_{0}(Y) \to KO^{[m+n]}_{0}(X \times Y)$
induced on Grothendieck-Witt groups of smooth schemes by the tensor product,
then we have $\times = \times '$.
\end{thm}

\begin{proof}
Suppose first that we have $\alpha \in \BO^{2i,i}(A)$ and $\beta \in \BO^{2j,j}(B)$ with $A$ and $B$
small pointed motivic spaces.  By Theorems \ref{T:varprojlim.a} and \ref{T:small.colim}
there exist $m$ and $n$ such that
$\Sigma_{T}^{4m+2-i}\alpha \in \BO^{8m+4,4m+2}(A \wedge T^{\wedge 2m})$
has a factorisation
\[
A \wedge T^{\wedge 4m+2} \to ([-n,n] \times HGr(n,2n),(0,x_{0})) \hra \ZZ \times HGr \cong
\KO^{[4m+2]} \to \BO(4m+2)
\]
and such that $\Sigma_{T}^{4m+2-j}\beta$ has a similar factorization.  Since $\Sigma_{T}^{4m+2-i}\alpha$
and $\Sigma_{T}^{4m+2-j}\beta$ are thus pullbacks of classes on which $\times$ and $\times'$ agree,
the products agree on $\Sigma_{T}^{4m+2-i}\alpha$
and $\Sigma_{T}^{4m+2-j}\beta$.
The compatibility of the products with the suspension and their $\veps$-commutativity imply that
we also have
$\Sigma_{T}^{8m+4-i-j}(\alpha \times \beta)  = \Sigma_{T}^{8m+4-i-j}(\alpha \times' \beta)$.
Since the suspension operation is a bijection on cohomology groups, we have
$\alpha \times \beta  = \alpha \times' \beta$.

For $\alpha \in \BO^{p,q}(A)$ and $\beta \in \BO^{p',q'}(B)$ with for instance
$p<2q$ and $p' >2q'$ the products
agree on
$\Sigma_{S^{1}_{s}}^{2q-p}\alpha \in \BO^{2q,q}(A)$ and
$\Sigma_{\GG_{m}}^{p'-2q'}\beta \in \BO^{2p'-2q',p'-q'}(B)$ by the previous case.  By the same sort of argument they also agree on $\alpha$ and $\beta$.  The other cases
are similar.
\end{proof}

\begin{thm}
The assertions of Theorem \ref{uniq2} and \ref{uniq1} hold.
\end{thm}

This follows from Theorems \ref {monoidBO} and \ref{T:unique.prod}.

\section{Spectra of finite and infinite real and quaternionic Grassmannians}
\label{S:finite}


We now wish to switch our sphere from $T$ to the pointed $HP^{1}$.  This is possible because
the pointed $HP^{1}$ is isomorphic to $T^{\wedge 2}$ in $H_{\bullet}(S)$.  The result is a spectrum
$\BO_{HP^{1}}$ very similar to $\BO$ except that the structural maps are induced by a product
with the Borel class of a symplectic bundle on a pointed scheme rather than a Thom class
on a Thom space.  This is important because the universal property of $\ZZ \times HGr$
deals with Grothendieck-Witt classes of bundles on schemes not of chain complexes on an $X$
acyclic over $U$.

The zigzag \eqref{E:T2=HP1} also gives us an equivalence between the stable homotopy categories
$SH_{T^{\wedge 2}}(S)$ and $SH_{(HP^{1},x_{\infty})}(S)$
\cite[Proposition 2.13]{Jardine:2000aa}.
There is also a Quillen equivalences between $Spt_{T}(S)$ and $Spt_{T^{\wedge 2}}(S)$
given by the forgetful functor and its adjoints.

\begin{thm}
The stable homotopy categories of $T$-spectra and of $(HP^{1},x_{0})$-spectra are
equivalent.
\end{thm}

The class $-b_{1}(\shf U_{HP^{1}},\phi_{HP^{1}}) \in
\KO_{0}^{[2]}(HP^{1},x_{\infty})$ corresponds to
$\thom^{\times 2} \in \KO_{0}^{[2]}(\Aff^{2}/\Aff^{2}-0)$ under the identifications.
So we may define our new spectrum.

\begin{defn}
\label{D:BO.HP1}
The $HP^{1}$-spectrum $\BO_{HP^{1}}$ corresponding to the $T$-spectrum $\BO$
has spaces
$(\KO^{[0]}, \KO^{[2]}, \KO^{[4]},\dots)$ and bonding maps
adjoint to the maps
\[
{-} \times -b_{1}(\shf U_{HP^{1}},\phi_{HP^{1}}) \colon \KO^{[2n]}({-}) \to
\KO^{[2n+2]}({-} \wedge (HP^{1},x_{\infty}))
\]
\end{defn}

The purpose of this section is to prove the following result.  For parallelism with the labelling of
real and complex Grassmannians, we write
$HGr'(2r,2n) = HGr(r,n)$ so that $RGr(2r,2n)$ and $HGr'(2r,2n)$ are both open subschemes
of the ordinary Grassmannian $Gr(2r,2n)$ where a certain bilinear form is nondegenerate.
We also write
\[
[-n,n]' = \{ i \in \ZZ \mid \text{$-n \leq i \leq n$ and $i \equiv n \bmod 2$} \}.
\]
For even $n$ the scheme $[-n,n]' \times RGr(n,2n)$ is pointed in the component corresponding
to $0 \in [-n,n]'$ by the point corresponding to
$\HH_{+}^{\oplus n/2} \oplus 0 \subset \HH_{+}^{\oplus n}$.  For odd $n$
we either do not use a base point or we use a disjoint base point.  To compactify our notations
we write
\begin{align}
HGr'_{2n} & = [-n,n] \times HGr'(2n,4n),
\\
RGr_{2n} & = ([-2n,2n]' \times RGr(2n,4n)) \cup ([-2n+1,2n-1]' \times RGr(2n-1,4n-2)).
\end{align}

The first step in the proof of this theorem is the following general construction.
For $(X,x)$ a pointed scheme over $S$, let $(X,x)^{+}$ be the pushout of
\begin{equation}
\label{E:X+}
\xymatrix @M=5pt @C=30pt {
\Aff^{1}
& \pt \ar[r]^-{x} \ar[l]^-{\sim}_-{0}
& X
}
\end{equation}
pointed by $1 \in \Aff^{1}(\pt)$.  (This is essentially the $\Aff^{1}$ mapping cone of the inclusion
$x \colon \pt \to X$.)
The natural projection $(X,x)^{+} \to (X,x)$ which is the identity on $X$ and sends $\Aff^{1} \to x$
is a motivic weak equivalence.

We abbreviate $HP^{1+} = (HP^{1},x_{0})^{+}$.  We will actually consider $HP^{1+}$-spectra.
The natural functor
$SH_{(HP^{1},x_{0})}(S) \to SH_{(HP^{1+}}(S)$ is an equivalence, and let $\BO_{HP^{1+}}$
be the $(HP^{1},x_{0})^{+}$-spectrum corresponding to $\BO$.

\begin{thm}
\label{T:finite}
There are $HP^{1+}$-spectra
$\BO^{\finite}$ and $\BO^{\geom}$ which are isomorphic to
$\BO_{HP^{1+}}$ in $SH_{HP^{1+}}(S)$
with spaces
\begin{gather}
\BO^{\finite}_{2i} =
\begin{cases}
RGr_{2\cdot 8^{i}} & \text{for even $i$}, \\
HGr'_{2\cdot 8^{i}} & \text{for odd $i$},
\end{cases}
\\
\BO^{\geom}_{2i} =
\begin{cases}
\ZZ \times RGr & \text{for even $i$}, \\
\ZZ \times HGr & \text{for odd $i$},
\end{cases}
\end{gather}
which are unions of real and quaternionic Grassmannians.
The bonding maps $\BO^{*}_{2i} \wedge HP^{1+} \to \BO_{2i+2}^{*}$ with $* = \finite$ or $\geom$
are morphisms of schemes and ind-schemes, respectively, which are constant on
$\BO^{*}_{2i} \vee HP^{1+}$.
\end{thm}

\begin{prop}
\label{P:A1.homotopy}
Let $(X,x)$, $(Y,y)$ and $(Z,z)$ be pointed schemes.  There is a natural bijection between
\begin{enumerate}
\item
the set of morphisms of pointed schemes
$g \colon (X,x) \times (Y,y)^{+} \to (Z,z)$ which restrict to the constant map
$\bigl( X \times  1 \bigr) \cup \bigl( x \times (Y,y)^{+} \bigr) \to z$, and
\item
the set of pairs
$(f,h)$ where
\begin{enumerate}
\item
$f \colon (X \times Y, x \times y) \to (Z,z)$ is a morphism of schemes
which restricts to the constant map $x \times Y \to z$ and
\item
$h \colon (X,x) \times \Aff^{1} \to (Z,z)$ is a pointed $\Aff^{1}$-homotopy between
$f |_{X \times y}$ and the constant map $X \to z$.
\end{enumerate}
\end{enumerate}
\end{prop}

The idea is that
$(X,x)\times (Y,y)^{+}$ is the union of two subschemes, and the restrictions
of $g$ to these subschemes are
$g |_{X \times Y}  = f$ and $g|_{X \times \Aff^{1}} = h$.
%



Let $U_{2n}$ and $U$ be the tautological symplectic subbundles on $HGr'(2n,4n)$ and
$HP^{1} = HGr'(2,4)$ resp., and let $V_{16n}$ be the tautological orthogonal subbundle
on $RGr(16n,24n)$.  Let $\HH$ be the trivial rank $2$ symplectic bundle, and let
$\HH_{+}$ be the trivial rank $2$ orthogonal bundle for the split quadratic form
$q(x_{1},x_{2}) = x_{1}x_{2}$.

\begin{lem}
\label{L:HGr.HP1.to.RGr}
There
exist morphisms of pointed schemes
\[
f_{2n} \colon \bigl( [-n,n] \times HGr'(2n,4n) \bigr) \times HP^{1} \to RGr(16n,32n)
\]
such that the Grothendieck-Witt classes satisfy
\begin{equation}
\label{E:GW.formula}
f_{2n}^{*}([V_{16n}]-8n[\HH_{+}]) = ([U_{2n}]-(n-i)[\HH]) \boxtimes ([U]-[\HH])
\end{equation}
\parens{where $i \in [-n,n] \subset \ZZ$ is the index of the component}
and such that
$f_{2n}|_{\pt \times HP^{1}}$
is constant, and
$f_{2n}|_{([-n,n] \times HGr'(2n,4n)) \times \pt}$
is pointed $\Aff^{1}$-homotopic to a constant map.
These maps and homotopies are
compatible with the inclusions $HGr'(2n,4n) \hra HGr'(2(n+1),4(n+1))$
and $RGr(16n,32n) \hra RGr(16(n+1),32(n+1))$.
\end{lem}

\begin{proof}
There are orthogonal direct sums $U_{2n} \oplus U_{2n}^{\perp} \cong \HH^{\oplus 2n}$
and $U \oplus U^{\perp} \cong \HH^{\oplus 2}$.  Consequently we have
\begin{equation*}
([U_{2n}]-(n-i)[\HH]) \boxtimes ([U]-[\HH])
=
[U_{2n} \boxtimes U]
+ [\HH^{\oplus n-i} \boxtimes U^{\perp}]
+ [U_{2n}^{\perp} \boxtimes \HH]
- (6n-2i) [\HH_{+}].
\end{equation*}
We have inclusions of orthogonal bundles
\begin{align*}
U_{2n} \boxtimes U
& \subset \HH^{\oplus 2n} \boxtimes U,
&
U_{2n}^{\perp} \boxtimes \HH
& \subset \HH^{\oplus 2n} \boxtimes \HH,
\\
\HH^{\oplus n-i} \boxtimes U^{\perp}
& \subset
\HH^{\oplus 2n} \boxtimes U^{\perp},
&
\HH_{+}^{\oplus 2n+2i}
& \subset
\HH_{+}^{\oplus 4n}
\end{align*}
Consequently the orthogonal subbundle
\begin{equation}
\label{E:subbundle}
(U_{2n} \boxtimes U)
\oplus (\HH^{\oplus n-i} \boxtimes U^{\perp})
\oplus (U_{2n}^{\perp} \boxtimes \HH)
\oplus \HH_{+}^{\oplus 2n+2i}
\end{equation}
of
\[
\bigl( \HH^{\oplus 2n} \boxtimes (U \oplus U^{\perp} \oplus \HH) \bigr)
\oplus \HH_{+}^{\oplus 4n} = \HH_{+}^{\oplus 16n}
\]
is classified by a map $f_{2n} \colon ([-n,n] \times HGr'(2n,4n)) \times HP^{1}
\to RGr(16n,32n)$.

When we restrict to $\pt \times HP^{1}$, we have $i=0$, and
the direct sum $U_{2n} \oplus U_{2n}^{\perp}$ becomes
$\HH^{\oplus n} \oplus \HH^{\oplus n}$.
The orthogonal direct sum
\eqref{E:subbundle} becomes
\[
\bigl( (\HH^{\oplus n} \oplus 0)\boxtimes (U \oplus U^{\perp} \oplus 0) \bigr)
\oplus
\bigl( ( 0 \oplus \HH^{\oplus n}) \boxtimes (0 \oplus 0 \oplus \HH) \bigr)
\oplus
(\HH_{+}^{\oplus 2n} \oplus 0).
\]
Since $U \oplus U^{\perp} \oplus 0 = \HH \oplus \HH \oplus 0$,
this is the same as
the subbundle
\begin{equation}
\label{E:pointing}
\bigl( (\HH^{\oplus n} \oplus 0)\boxtimes (\HH \oplus \HH \oplus 0) \bigr)
\oplus
\bigl( ( 0 \oplus \HH^{\oplus n}) \boxtimes (0 \oplus 0 \oplus \HH) \bigr)
\oplus
(\HH_{+}^{\oplus 2n} \oplus 0).
\end{equation}
corresponding to the pointing of $RGr(16n,32n)$.  So
$f_{2n}(\pt \times HP^{1}) = \pt$.

A similar argument shows that $f_{2n}$ is compatible with the inclusions
of the Grassmannians in higher-dimensional Grassmannians.

When we restrict $f_{2n}$ to $([-n,n] \times HGr'(2n,4n)) \times \pt$,
the direct sum over $\pt \subset HP^{1}$ becomes $U \oplus U^{\perp} = \HH \oplus \HH$.
The orthogonal direct sum
\eqref{E:subbundle} becomes
\begin{equation}
\label{E:homotopic.subbundle}
\bigl(U_{2n}  \boxtimes (\HH \oplus 0 \oplus 0) \bigr)
\oplus \bigl(\HH^{\oplus n-i} \boxtimes (0 \oplus \HH \oplus 0) \bigr)
\oplus \bigl(U_{2n}^{\perp} \boxtimes(0 \oplus 0 \oplus \HH) \bigr)
\oplus \HH_{+}^{\oplus 2n+2i}
\end{equation}
This direct sum decomposition of the subbundle of rank $16n$
is compatible with the orthogonal direct sum decomposition
of the bundle of rank $32n$ into two summands
\begin{equation}
\label{E:first.summand}
\HH^{\oplus 2n} \boxtimes (\HH \oplus 0 \oplus \HH)
= (U_{2n} \oplus U_{2n}^{\perp}) \boxtimes(\HH \oplus 0 \oplus \HH)
= (\HH^{\oplus n} \oplus \HH^{\oplus n}) \boxtimes(\HH \oplus 0 \oplus \HH)
\end{equation}
and
\begin{equation}
\label{E:second.summand}
\bigl( \HH^{\oplus 2n} \boxtimes (0 \oplus \HH \oplus 0) \bigr)
\oplus \HH_{+}^{4n} =
\HH_{+}^{8n}.
\end{equation}
Because
\[
M_{1} =
\begin{pmatrix}
0 & 0 & -1 & 0 \\
0 & 0 & 0 & -1 \\
1 & 0 & 0 & 0 \\
0 & 1 & 0 & 0
\end{pmatrix}
=
\begin{pmatrix}
1 & 0 & 0 & 0 \\
0 & 1 & 0 & -1 \\
1 & 0 & 1 & 0 \\
0 & 0 & 0 & 1
\end{pmatrix}
\begin{pmatrix}
1 & 0 & -1 & 0 \\
0 & 1 & 0 & 0 \\
0 & 0 & 1 & 0 \\
0 & 1 & 0 & 1
\end{pmatrix}
\begin{pmatrix}
1 & 0 & 0 & 0 \\
0 & 1 & 0 & -1 \\
1 & 0 & 1 & 0 \\
0 & 0 & 0 & 1
\end{pmatrix}
\]
is a product of matrices which are elementary symplectic and elementary orthogonal, there is a morphism $M \colon \Aff^{1} \to Sp_{4} \cap O_{4}$
with $M(0) = I$ and $M(1) = M_{1}$, namely
\[
M(t) =
\begin{pmatrix}
1-t^{2} & 0 & -t & 0 \\
0 & 1-t^{2} & 0 & -2t+t^{3} \\
2t-t^{3} & 0 & 1-t^{2} & 0 \\
0 & t & 0 & 1-t^{2}
\end{pmatrix}
\]
When we restrict to $\pt \times \pt$ the direct sum $U \oplus U^{\perp}$ becomes $\HH \oplus \HH$.
We may see that
\[
\bigl( (1_{\HH^{\oplus n} \oplus 0} \boxtimes 1_{\HH^{\oplus 2}})
\oplus (1_{0 \oplus \HH^{\oplus n}} \boxtimes M )\bigr)
^{-1}
\bigl( (1_{U_{2n}} \boxtimes 1_{\HH^{\oplus 2}})
\oplus (1_{U_{2n}^{\perp}} \boxtimes M )\bigr)
\]
is a pointed $\Aff^{1}$-homotopy between the subbundle
\[
\bigl(U_{2n}  \boxtimes (\HH \oplus 0 \oplus 0) \bigr)
\oplus \bigl(U_{2n}^{\perp} \boxtimes(0 \oplus 0 \oplus \HH) \bigr)
\]
and the subbundle
\[
\bigl( (\HH^{\oplus n} \oplus 0)\boxtimes (\HH \oplus 0 \oplus 0) \bigr)
\oplus
\bigl( ( 0 \oplus \HH^{\oplus n}) \boxtimes (0 \oplus 0 \oplus \HH) \bigr)
\]
One may construct a similar pointed $\Aff^{1}$-homotopy between the subbundles
\[
\bigl( (\HH^{\oplus n-i} \oplus 0) \boxtimes ( 0 \oplus \HH \oplus 0) \bigr)
\oplus (\HH_{+}^{\oplus 2n+2i} \oplus 0)
\]
and
\[
\bigl( (\HH^{\oplus n} \oplus 0) \boxtimes ( 0 \oplus \HH \oplus 0) \bigr)
\oplus (\HH_{+}^{\oplus 2n} \oplus 0)
\]
Combining the two gives a pointed $\Aff^{1}$-homotopy between \eqref{E:homotopic.subbundle}
and \eqref{E:pointing} and thus between $f_{2n} |_{([-n,n] \times HGr'(2n,4n)) \times \pt}$ and the
constant map.

The compatibility of the homotopies with the inclusions of Grassmannians is relatively straightforward.
\end{proof}

The next lemma is proven in the same way as the last one.

\begin{lem}
\label{L:RGr.HP1.to.HGr}
There
exist morphisms of pointed schemes
\[
g_{n} \colon ([-n,n]' \times RGr(n,2n)) \times HP^{1} \to HGr'(8n,16n)
\]
such that the Grothendieck-Witt classes satisfy
\begin{equation}
\label{E:GW.formula.2}
g_{n}^{*}([U_{8n}]-4n[\HH]) = ([V_{n}]-\tfrac{1}{2}(n-i)[\HH_{+}]) \boxtimes ([U]-[\HH])
\end{equation}
\parens{where $i \in [-n,n]' \subset \ZZ$ is the index of the component}
and such that
$g_{n}|_{\pt \times HP^{1}}$ is constant,
and $g_{n}|_{([-n,n]' \times RGr(n,2n)) \times \pt}$
is
pointed $\Aff^{1}$-homotopic to a constant map.  These maps and homotopies are
compatible with the inclusions $RGr(n,2n) \hra RGr(n+2,2n+4)$
and $HGr'(8n,16n) \hra HGr'(8(n+2),16(n+2))$.
\end{lem}

\begin{proof}[Proof of Theorem \ref{T:finite}]
By Proposition \ref{P:A1.homotopy} the maps and homotopies of
Lemmas \ref{L:HGr.HP1.to.RGr} and \ref{L:RGr.HP1.to.HGr}
give us maps
\begin{align*}
F_{n} \colon & HGr'_{2n}  \wedge HP^{1+} \to RGr_{16n},
\\
G_{n} \colon & RGr_{2n} \wedge HP^{1+} \to HGr'_{16n}.
\end{align*}

We now define $\BO^{\finite}$ to be the $HP^{1}$-spectrum with spaces as in the
statement of Theorem \ref{T:finite} and with bonding maps the compositions
\begin{gather*}
RGr_{2 \cdot 8^{2i}} \wedge HP^{1+}
\to
HGr'_{2 \cdot 8^{2i+1}}
\to
HGr'_{2 \cdot 8^{2i+2}}
\\
HGr'_{2 \cdot 8^{2i}} \wedge HP^{1+}
\to
RGr_{2 \cdot 8^{2i+1}}
\to
RGr_{2 \cdot 8^{2i+2}}
\end{gather*}
of the appropriate
$F_{n}$ or $G_{n}$ with the maps induced by the inclusions of Grassmannians.

We define $\BO^{\geom}$ to be the $HP^{1}$-spectrum with spaces
\[
\BO^{\geom}_{2i} =
\begin{cases}
\colim_{n} RGr_{2n}  = \ZZ \times RGr & \text{for even $i$}, \\
\colim_{n} HGr'_{2n} = \ZZ \times HGr & \text{for odd $i$},
\end{cases}
\]
and with bonding maps induced by the $F_{n}$ and $G_{n}$.

We claim that the inclusion map $\BO^{\finite} \to \BO^{\geom}$
is a stable weak equivalence.  To show this we need to show that the maps
$\colim_{i} \Omega_{HP^{1}}^{i} (\BO^{\finite}_{2i+2j})^{f} \to
\colim_{i} \Omega_{HP^{1}}^{i} (\BO^{\geom}_{2i+2j})^{f}$ are weak equivalences for all $j$.
The $(-)^{f}$ denotes fibrant replacement.
This is because we have two $\NN^{2}$-indexed families of spaces
\begin{align*}
E_{n,i} & =
\begin{cases}
\Omega_{HP^{1}}^{i} (RGr_{2\cdot 8^{i}})^{f} & \text{for even $i$}, \\
\Omega_{HP^{1}}^{i} (HGr'_{2\cdot 8^{i}})^{f} & \text{for odd $i$},
\end{cases}
\\
E'_{n,i} & =
\begin{cases}
\Omega_{HP^{1}}^{i} (HGr'_{2\cdot 8^{i}})^{f} & \text{for even $i$}, \\
\Omega_{HP^{1}}^{i} (RGr_{2\cdot 8^{i}})^{f} & \text{for odd $i$},
\end{cases}
\end{align*}
The inclusions of Grassmannians the $F_{n}$ and $G_{n}$ give us maps $E_{n,i} \to E_{n+1,i}$
and $E_{n,i} \to E_{n+1,i+1}$ which commute, and similarly for the $E'_{n,i}$.
Thus the $E_{n,i}$ and $E'_{n,i}$ are filtered systems of spaces
indexed by a category with set of objects $\NN^{2}$ such that there is a unique
arrow $(n,i) \to (n',i')$ if and only if $n \leq n'$ and $i-n \leq i'-n'$.  By cofinality we have
isomorphisms $\colim_{i} E_{i,2i+2j} = \colim_{n,i} E_{n,i} = \colim_{i} (\colim_{n} E_{n,i})$ for all $j$
and similarly for $E'_{n,i}$.  These are the required weak equivalences for even and odd $j$
respectively.

We now construct an isomorphism $\BO^{\geom} \cong \BO_{HP^{1+}}$
in $SH_{HP^{1+}}(S)$.
By Theorem \ref{T:varprojlim.b} we have an exact sequence
\begin{equation*}
0 \to \varprojlim\nolimits^{1} \BO_{HP^{1+}}^{4n-1,2n}(\BO^{\geom}_{2n})
\to \BO^{0,0}_{HP^{1+}}(\BO^{\geom}) \to
\varprojlim \BO^{4n,2n}_{HP^{1+}}(\BO^{\geom}_{2n})
\to 0.
\end{equation*}
For every odd $n = 2k+1$ the universal element of Definition \ref{D:tau} gives an isomorphism
$\tau_{4k+2} \colon \ZZ \times HGr \cong \KO^{[4k+2]}$ in $H_{\bullet}(S)$
and by adjunction a
$\tau_{4k+2}' \in \BO^{8k+4,4k+2}(\BO^{\geom}_{4k+2})$.  The inverse system
sends $\tau_{4k+2}' \mapsto \tau_{4k-2}'$ because in the diagram
\begin{equation}
\label{E:compatible}
\vcenter{
\xymatrix @M=5pt @C=70pt {
(\ZZ \times HGr) \wedge HP^{1+} \wedge HP^{1+} \ar[r]
\ar[d]_-{\tau_{4k-2} \wedge 1 \wedge 1}^-{\cong}
& \ZZ \times HGr \ar[d]_-{\tau_{4k+2}}^-{\cong}
\\
\KO^{[4k-2]} \wedge HP^{1+}\wedge HP^{1+} \ar[r]
& \KO^{[4k+2]}
}}
\end{equation}
both horizontal maps are a $\times$ product with $b_{1}(U)^{\times 2} = ([U]-[\HH])^{\times 2}$.
So the $\tau'_{4k+2}$ define an inverse system
$\tau'\in \varprojlim \BO^{8k+4,4k+2}_{HP^{1}}(\BO^{\geom}_{4k+2})$.  Since
the $\tau'_{4k+2}$ are all isomorphisms, any element of
$\BO^{0,0}_{HP^{1}}(\BO^{\geom})$
lifting $\tau'$ is an isomorphism by general facts about homotopy colimits of sequential direct systems
in triangulated categories.
\end{proof}

The inverse system $\tau'$ also lies in
$\varprojlim \BO^{4i,2i}_{HP^{1}}(\BO^{\geom}_{2i})$.  So it also gives
us maps $\tau_{4k} \colon \ZZ \times RGr \to \KO^{[4k]}$ in $H_{\bullet}(S)$.  Essentially
these are the compositions of $\tau_{4k+2}$ and the maps in the motivic unstable homotopy
category induced by the adjoint bonding maps of the two spectra
\[
\ZZ \times RGr \to \Omega_{HP^{1+}}(\ZZ \times HGr)
\xra{\sim} \Omega_{HP^{1+}}\KO^{[4k+2]} \xla{\sim} \KO^{[4k]}.
\]
We do not know whether these are isomorphisms in $H_{\bullet}(S)$.  The best we know how
to do is:

\begin{prop}
\label{P:right.inverse}
The morphism $\Omega_{HP^{1}}(\tau_{4k})$ has a right inverse in $H_{\bullet}(S)$.
\end{prop}

This is because in the commutative diagram
\[
\xymatrix @M=5pt {
\ZZ \times HGr \ar[r] \ar[d]^{\sim}
&
\Omega_{HP^{1+}}(\ZZ \times RGr) \ar[d]
\\
\KO^{[4k-2]} \ar[r]^-{\sim}
&
\Omega_{HP^{1+}}\KO^{[4k]}
}
\]
the arrows on the left side and bottom of the square are weak equivalences by Theorem \ref{T:MV.symp}
and \ref{OmegaSpectrumBO} respectively.

\section{The commutative monoid structure in $SH(S)$ }
\label{S:vanishing}

The main technical result of this section is the following theorem.  We use it to show that
the almost commutative monoid structure on the $T$-spectrum $\BO$ we constructed in
Theorem \ref {monoidBO} is actually a commutative monoid for $S = \Spec \ZZ[\frac 12]$.

\begin{thm}
\label{T:lim1}
Let $S$ be a regular noetherian separated $\ZZ[\frac 12]$-scheme of finite Krull dimension.
Suppose
that $KO_{1}(S)$ and $KSp_{1}(S)$ are finite.
Then for all $m$
the natural map
\[
\BO^{0,0}(\BO^{\wedge m}) \to
\varprojlim\BO^{2mi,mi}((\KO^{[i]})^{\wedge m})
\]
is an isomorphism.
\end{thm}

\begin{proof}

We prove the theorem for the $HP^{1}$-spectrum $\BO_{HP^{1}}$.  The theorem then follows
for the $T$-spectrum $\BO$.

By Theorem \ref{T:varprojlim.b} and Theorem \ref{T:finite} there is a commutative diagram with exact rows
\[
\xymatrix @M=5pt @C=12pt @R=18pt {
\raisebox{-18pt}{$\sideset{}{^{1}}\varprojlim\limits_{i}\BO_{HP^{1}}^{4mi-1,2mi}((\KO^{[2i]})^{\wedge m})$}
\ar@{>->}[r] \ar[d]
&
\BO_{HP^{1}}^{0,0}(\BO_{HP^{1}}^{\wedge m})
\ar@{->>}[r] \ar[d]^-{\cong}
&
\raisebox{-18pt}{$\varprojlim\limits_{i}\BO_{HP^{1}}^{4mi,2mi}((\KO^{[2i]})^{\wedge m})$}
\ar[d]^{\text{$\cong$ by cofinality}}
\\
\raisebox{-18pt}{$\sideset{}{^{1}}\varprojlim\limits_{i}\BO_{HP^{1}}^{4mi-1,2mi}((\BO^{\finite}_{2i})^{\wedge m})$} \ar@{>->}[r]
&
\BO^{0,0}_{HP^{1}}((\BO^{\finite})^{\wedge m}) \ar@{->>}[r]
&
\raisebox{-18pt}{$\varprojlim\limits_{i}\BO^{4mi,2mi}_{HP^{1}}((\BO^{\finite}_{2i})^{\wedge m})$}
}
\]
The middle vertical arrow is an isomorphism because
the morphism
$(\BO^{\finite}_{2i})^{\wedge m} \to \BO_{HP^{1}}^{\wedge m}$ is a stable
weak equivalence.  The righthand vertical arrow is an isomorphism by a cofinality argument
as in the proof of Theorem \ref{T:finite}.  It follows that the lefthand vertical arrow is also an
isomorphism.

The $\varprojlim^{1}$ in the lower row is isomorphic to
\[
\sideset{}{^{1}}\varprojlim\limits_{\text{$i$ odd}} KO_{1}(HGr'_{2N_{i}}{}^{\wedge m})
\qquad \text{or} \qquad
\sideset{}{^{1}}\varprojlim\limits_{\text{$i$ odd}} KSp_{1}(HGr'_{2N_{i}}{}^{\wedge m})
\]
for even $m$ and odd $m$, respectively, where $N_{i} = 8^{2i}$.  By Theorem \ref {T:Grass}
each group in the system
is a direct sum of a finite number of copies of $KO_{1}(S)$ and of $KSp_{1}(S)$.  By the hypothesis
it
follows that each group in this system is finite, and so the $\varprojlim^{1}$ in the lower row
of the diagram vanishes.
Therefore the $\varprojlim^{1}$ in the upper row also vanishes, proving the isomorphism.
%
%
%
\end{proof}

\begin{thm}
\label{T:KSp1.euclidean}
Let $R$ be a Euclidean domain.  Then we have $KSp_{1}(R) = 0$.
\end{thm}

This is classical.  It is proven essentially
by showing that the action of the group $ESp_{2n}(R)$
on unimodular vectors is transitive.

\begin{thm}
\label{T:pid}
Let $R$ be a Euclidean domain with $\frac 12 \in R$.
%
%
Then we have
$KO_{1}(R) \cong \ZZ/ 2 \ZZ \times R^{\times}/R^{\times 2}$.

\end{thm}

This must be very well known to the experts.  We include the proof for completeness' sake.

\begin{proof}
[Proof of Theorem \ref{T:pid}]
We use the long exact sequences of Karoubi's fundamental theorem
\cite{Karoubi:1980aa,Schlichting:2006aa}
\[
\cdots \to KO_{i}^{[n]}(R) \xrightarrow{F} K_{i}(R) \xrightarrow{H}
KO_{i}^{[n+1]}(R) \xrightarrow{\eta} KO_{i-1}^{[n]}(R) \to \cdots.
\]
with $F$ the forgetful map and $H$ the hyperbolic map.  This amounts to four exact sequences including
\begin{gather}
\label{E:fund.3}
\cdots \to
{}_{-1}V \to K_{1} \to KO_{1} \to {}_{1}U \xra{0} K_{0} \hra GW^{+} \to W^{0} \to 0
\\
\label{E:fund.4}
\cdots \to
KSp_{1} \to K_{1} \onto {}_{-1}V \xra{0} GW^{-} \hra K_{0} \to {}_{1}U \to W^{-1} \to 0
\end{gather}
with the $GW^{+}$ and $GW^{-}$ the Grothendieck-Witt groups of symmetric and skew-symmetric bilinear forms respectively,
the ${}_{-1}V$ and ${}_{1}U$ the groups of \cite[Appendices 2 et 3]{Karoubi:1975aa},
and the $W^{i}$ the Witt groups with cohomological indexing \textit{\`a la} Balmer.


For a principal ideal domain containing $\frac 12$ we have $W^{i} = 0$ for $i \equiv 2$ or $3 \pmod 4$,
while the
the forgetful map $GW^{-} \to K_{0}$ of \eqref{E:fund.4}
is the inclusion $2\ZZ \subset \ZZ$, and the hyperbolic map $K_{0} \to GW^{+}$
of \eqref{E:fund.3} is injective.  It then follows from the two sequences that we have a short exact sequence
\[
0 \to \coker(K_{1} \onto {}_{-1}V \to K_{1}) \to KO_{1} \to \ZZ/2\ZZ \to 0.
\]
By \cite[\S 4.1]{Karoubi:1980aa} or \cite[\S 4.5]{Barge:2008it} the composition
$K_{1} \onto {}_{-1}V \to K_{1}$ is
$[x] \mapsto [x]-[{}^{t}\bar{x}^{-1}]$.  For a euclidean domain with trivial involution, this
is $[x] \mapsto [x^{2}]$.  So we have an exact sequence
$1 \to R^{\times}/R^{\times 2} \to KO_{1} \to \ZZ/2\ZZ \to 0$.
The image of $[b] \in R^{\times}/R^{\times 2}$ in $KO_{1}$
is the class of the auto-isometry $\bigl[\begin{smallmatrix} b & 0 \\ 0 & b^{-1} \end{smallmatrix}\bigr]$ of the
hyperbolic quadratic form $q(x_{1},x_{2}) = x_{1}x_{2}$.
The class of the auto-isometry $\bigl[\begin{smallmatrix} 0&1\\1&0 \end{smallmatrix}\bigr]$ is not in that image
because its determinant is $-1$.  It provides a splitting of the surjection in the exact sequence.
\end{proof}

For a nice discussion of the group ${}_{1}V$ and a bit of the other two exact sequences of Karoubi's
fundamental theorem see \cite[\S 4.5]{Barge:2008it}.

\begin{thm}
\label{T:commutative.monoid}
Suppose $KO_{1}(S)$ and $KSp_{1}(S)$ are finite, for instance $S = \Spec \ZZ[\frac 12]$.
Let $m \in Hom_{SH(S)}(\BO \wedge \BO , \BO)$ be the morphism of \eqref{E:m}.
Let $e \in Hom_{SH(S)}(\pt_{+},\BO) = \BO^{0,0}(\pt_{+})$ be the element corresponding
to $\angles{1} \in GW^{+}(\pt) = KO_{0}^{[0]}(\pt)$.

\parens{a} Then $(\BO,m,e)$ is a commutative monoid in $SH(S)$.

\parens{b} The map $m$ is the unique element of $Hom_{SH(S)}(\BO \wedge \BO , \BO)$ defining
a pairing which, when restricted to pairing
\[
\BO^{4p,2p}(X_{+}) \times \BO^{4q,2q}(Y_{+}) \to
\BO^{4p+4q,2p+2q}(X_{+} \wedge Y_{+})
\]
with $X,Y \in \Sm/S$ coincides with the tensor product pairing
\[
KO_{0}^{[2p]}(X) \times KO_{0}^{[2q]}(Y)
\to KO_{0}^{[2p+2q]}(X \times Y)
\]
of Grothendieck-Witt groups.
\end{thm}

\begin{proof}
(a) By Theorem \ref{monoidBO} $(B,m,e)$ is an almost commutative monoid in $SH(S)$.
By Definition \ref{D:almost.monoid} the obstructions to $(B,m,e)$ being a commutative monoid
are three classes in the kernels of the maps of Theorem \ref{T:lim1} for $m=1,2,3$.  Those
classes vanish.

(b) The product on the $KO_{0}^{[2p]}(X_{+})$ determined uniquely the
$\bar m$ of \eqref{E:bar.m}.  By Theorem \ref{T:lim1} $m$ is the unique element of
$Hom_{SH(S)}(\BO \wedge \BO , \BO)$ mapping onto $\bar m$.
\end{proof}

We now wish to use the closed motivic model structure
of \cite[Appendix A]{Panin:2009aa}.
Among its properties are:

\begin{enumerate}
\item
The closed motivic model structure
and the local injective motivic model structure
used in \S\S \ref{S:KO.motivic.spaces}--\ref{S:BO}
have the same weak equivalences, but the closed motivic model structure
has fewer cofibrations
and more fibrations than the local injective motivic model structure.

\item
A pointed smooth $S$-scheme $(X,x_{0})$ is cofibrant in the closed motivic model structure.
More generally, a closed embedding $Z \mono X$ in $\Sm/S$ induces a cofibration
$Z_{+} \mono X_{+}$
\cite[Lemma A.10]{Panin:2009aa}.

Hence the pointed scheme $HP^{1+}$ is cofibrant for the closed motivic model structure, so we may
define levelwise and stable closed motivic model structures for $HP^{1+}$-spectra.

\item
For any morphism $u \colon S\to S'$ of noetherian schemes of finite Krull dimension,
the pullback $u^{*}\colon \M_{\bullet}(S') \to \M_{\bullet}(S)$ is a
strict symmetric monoidal left Quillen functor for the closed motivic model structure
\cite[Theorem A.17]{Panin:2009aa}.
Consequently $\mathbf{L}u^{*} \colon SH(S') \to SH(S)$ can be computed by
taking  levelwise closed cofibrant replacements and then applying $u^{*}$.

\end{enumerate}

To extend $m$ to other base schemes $S$, we need to discuss base change for morphisms
$u \colon S \to S'$.  For any $X \in \Sm/S'$ there is a duality-preserving pullback functor inducing
morphisms of hermitian $K$-theory spaces
$(1 \times u)^{*} \colon KO^{[n]}(X) \to KO^{[n]}(X \times_{S'}S)$.
This gives us maps $KO^{[n]}_{S'} \to u_{*}KO^{[n]}_{S}$ and adjoint maps $u^{*} KO_{S'}^{[n]} \to
KO_{S}^{[n]}$.  These maps are compatible with Thom isomorphisms, inducing maps of spectra.
The maps $u^{*}\BO^{\geom}_{S'} \to \BO^{\geom}_{S}$ are isomorphisms in
$SH(S)$ because $u^{*}$ acts as base change on the quaternionic and real Grassmannians and
their direct colimits.
The maps $\mathbf{L}u^{*}\BO^{\geom}_{S'} \to u^{*}\BO^{\geom}_{S'}$
are isomorphisms $SH(S)$ because for the closed motivic model structure
$\BO^{\geom}_{S'}$ is levelwise cofibrant
and $u^{*}$ is a levelwise left Quillen functor.

Setting $S' = \Spec \ZZ[\frac 12]$ with $u \colon S \to \Spec \ZZ[\frac 12]$
the canonical map, we can now define the monoidal structure on $\BO_{S}$ in $SH(S)$
as in \cite[Definition 3.7]{Panin:2009aa} as the composition
\[
m_{S} \colon
\BO_{S} \wedge \BO_{S} \cong
u^{*}\BO_{\ZZ[\frac 12]} \wedge u^{*}\BO_{\ZZ[\frac 12]}
\cong
u^{*}(\BO_{\ZZ[\frac 12]} \wedge \BO_{\ZZ[\frac 12]})
\xra{u^{*}m_{\ZZ[\frac 12]}}
u^{*}\BO_{\ZZ[\frac 12]}
\cong
\BO_{S}
\]

\begin{thm}
\label{T:unique.2}
The assertions of Theorem \ref{T:unique} hold.
\end{thm}

For $S = \Spec \ZZ[\frac 12]$ this is part of Theorem \ref{T:commutative.monoid}.
For other $S$ it is deduced by base change from $\Spec \ZZ[\frac 12]$.

\begin{thm}
The assertions of Theorem \ref{T:SLc.ring} hold.
\end{thm}

Theorem \ref{T:SLc.oriented} shows that hermitian $K$-theory is an $SL^{c}$-oriented cohomology
theory with a partial multiplicative structure.  The ring structure is given by Theorem
\ref{T:unique.2}.  The compatibility of the two multiplications is Theorem \ref{T:coincide}.

Schlichting's multiplicative structure, which we mentioned when discussing Theorem \ref{T:compatible},
could replace our partial multiplicative structure for
Theorems \ref{T:SLc.ring}, \ref{uniq2}, \ref{uniq1}, etc.
However, as we understand it,
Schlichting's product is defined in unstable homotopy theory.
To get our main Theorem \ref{T:unique}
with the monoid structure for $T$-spectra, we need our argument with the $\varprojlim^{1}$.

\section{
{$CP^{1+}$-spectra $\BGL^{\finite}$ and $\BGL^{\geom}$ for algebraic $K$-theory}
{CP\^{ }1+ spectra BGL\^{ }fin and BGL\^{ } for algebraic K-theory}
}
\label{S:K.theory}

The $HP^{1+}$-spectra
constructed in \S\ref{S:finite} have
an analogue for ordinary algebraic $K$-theory: the $CP^{1}$-spectra $\BGL^{\finite}$
and $\BGL^{\geom}$.
We sketch their construction.  The first can be used to show that the uniqueness results concerning
the algebraic $K$-theory spectrum $\BGL$ and its $\times$ product
of \cite[Remark 2.19 and Theorem 3.6]{Panin:2009aa} hold for any base scheme $S$
which is noetherian of finite Krull dimension with finite $K_{1}(S)$ and not just
for $S = \Spec \ZZ$.

We use the affine Grassmannians which can be defined as
\[
CGr(m,n) = GL_{n}/(GL_{m} \times GL_{n-m})
\]
or as the open subscheme
\[
CGr(m,n) \subset Gr(m,n) \times Gr(n-m,n)
\]
where the two tautological subbundles of $\OO^{\oplus n}$ are supplementary
or as the closed subscheme of the space on $n \times n$ matrices parametrizing projectors of rank $m$.
Each $CGr(m,n)$  is affine over the base scheme and
an $\Aff^{m(n-m)}$-bundle over the ordinary Grassmannnian
$Gr(m,n)$.  Morphisms $V \to CGr(m,n)$ are in bijection with direct sum decompositions
$\OO_{V}^{\oplus n} = U'_{m} \oplus U''_{n-m}$ with $U'_{m}$ and $U''_{n-m}$ subbundles
of ranks $m$ and $n-m$ respectively.  We let $CGr = \colim_{n} CGr(n,2n)$.

In particular
$CP^{1} = CGr(1,2) \cong \mathbf{P}^{1} \times \mathbf{P}^{1} - \Delta$ is an $\Aff^{1}$-bundle over
$\mathbf P^{1}$.  We may point $CGr(1,2)$ by $CGr(0,0)$.  Let $CP^{1+}$ then be the pointed scheme
constructed in \eqref{E:X+}.
The motivic stable homotopy categories of $\mathbf{P}^{1}$-spectra,
of $CP^{1}$-spectra and of $CP^{1+}$-spectra are equivalent.  In particular there is a $CP^{1+}$-spectrum
$\BGL_{CP^{1+}}$ corresponding to the $\mathbf{P}^{1}$-spectrum $\BGL$
of \cite{Panin:2009aa}.  For any smooth $S$-scheme $X$ we write $n = n[\OO_{X}] \in K_{0}(X)$.

\begin{lem}
\label{L:Gr.P1.to.Gr}
There
exist morphisms of pointed schemes
\[
h_{n} \colon \bigl( [-n,n] \times CGr(n,2n) \bigr) \times CP^{1} \to CGr(4n,8n)
\]
such that the classes in $K_{0}$ satisfy
\begin{equation}
\label{E:K0.formula}
h_{n}^{*}([U'_{4n}]-4n) = ([U'_{n}]-(n-i)) \boxtimes ([U'_{1}]-1)
\end{equation}
\parens{where $i \in [-n,n] \subset \ZZ$ is the index of the component}
and such that
$h_{n}|_{\pt \times CP^{1}}$ is constant, and
$h_{n}|_{([-n,n] \times CGr(n,2n)) \times \pt}$ is
pointed $\Aff^{1}$-homotopic to a constant map.  Moreover, these maps and homotopies are
compatible with the inclusions $CGr(n,2n) \hra CGr(n+1,2(n+1))$
and $CGr(4n,8n) \hra CGr(4(n+1),8(n+1))$.
\end{lem}

This lemma is proven in the same way as Lemma \ref{L:HGr.HP1.to.RGr} using the equality
\[
([U'_{n}]-(n-i)) \boxtimes ([U'_{1}]-1) = [U'_{n} \boxtimes U'_{1}] + [\OO^{\oplus n-i}\boxtimes U''_{1}]
+ [U''_{n} \boxtimes \OO] - (3n-i)[\OO \boxtimes \OO]
\]
in $K_{0}([-n,n] \times CGr(n,2n))$
and the direct sum decompositions of vector bundles
\begin{gather*}
(U'_{n} \boxtimes U'_{1}) \oplus (U''_{n} \boxtimes U'_{1})
 = \OO^{\oplus 2n} \boxtimes U'_{1},
\\
(\OO^{\oplus n-i} \boxtimes U''_{1}) \oplus (\OO^{\oplus n+i} \boxtimes U''_{1})
 =
\OO^{\oplus 2n} \boxtimes U''_{1},
\\
(U''_{n} \boxtimes \OO) \oplus (U'_{n} \boxtimes \OO)
 = \OO^{\oplus 2n} \boxtimes \OO,
\\
(\OO^{\oplus n+i} \boxtimes \OO) \oplus (\OO^{\oplus n-i} \boxtimes \OO)
 =
\OO^{\oplus 2n} \boxtimes \OO,
\end{gather*}
yielding a decomposition of the trivial bundle of  rank $8n$ on $([-n,n] \times CGr(n,2n)) \times CP^{1}$
as the direct sum of two
subbundles of rank $4n$.

\begin{thm}
\label{T:BGL.finite}
There are $CP^{1+}$-spectra
$\BGL^{\finite}$ and $\BGL^{\geom}$ isomorphic to
$\BGL_{CP^{1}}$ in $SH_{CP^{1}}(S)$
with spaces
\begin{align*}
\BGL^{\finite}_{n} & =
[-4^{n},4^{n}] \times CGr(4^{n}, 2 \cdot 4^{n})
&
\BGL^{\geom}_{n} & = \mathbb Z \times CGr
\end{align*}
which are unions of affine Grassmannians.
The bonding maps $\BGL^{*}_{n} \wedge CP^{1+} \to \BGL_{n+1}^{*}$
of the two spectra
are morphisms of schemes or ind-schemes which are constant on the wedge
$\BGL^{*}_{n} \vee CP^{1+}$.
\end{thm}

This theorem is proven in the same way as Theorem \ref{T:finite}.


\end{document}